\documentclass[11pt,reqno]{amsart}
\usepackage{graphicx}
\usepackage{amsfonts,amsmath,amssymb}
\newcommand{\rl}{{\mathbb{R}}}
\newcommand{\cx}{{\mathbb{C}}}

\newcommand{\e}{\varepsilon}

\newcommand{\tmop}[1]{\ensuremath{\operatorname{#1}}}
\renewcommand{\Re}{\tmop{Re}}

\renewcommand {\a}{\alpha}

\newtheorem{theorem}{Theorem}
\newtheorem{lemma}{Lemma}
\newtheorem{prop}{Proposition}

\newtheorem{cor}{Corollary}

\begin{document}
\title{Local polynomial convexity of the unfolded Whitney umbrella in $\mathbb C^2$}
\author{Rasul Shafikov* and Alexandre Sukhov**}
\begin{abstract}
The paper considers a class of Lagrangian surfaces in $\cx^2$ with isolated singularities of the unfolded 
Whitney umbrella type. We prove that generically such a surface is locally polynomially convex near a 
singular point of this kind.
\end{abstract}

\maketitle

MSC: 32E20, 32E30, 32V40, 53D12.

\bigskip

Key words: totally real manifold, Lagrangian manifold, Whitney umbrella, polynomial convexity,  characteristic foliation, dynamical system, Newton diagram.

\bigskip
* Department of Mathematics, the University of Western Ontario, London, Ontario, N6A 5B7, Canada,
e-mail: shafikov@uwo.ca. The author is partially supported by the Natural Sciences and Engineering 
Research Council of Canada.

**Universit\'e des Sciences et Technologies de Lille, 
U.F.R. de Math\'ematiques, 59655 Villeneuve d'Ascq, Cedex, France,
e-mail: sukhov@math.univ-lille1.fr

\section{Introduction}\label{s:1}

Polynomial convexity of real submanifolds of $\cx^n$ is a well-studied subject in complex analysis due to its 
deep relation to  the approximation problems, pluripotential theory  and Banach algebras (see, for instance, 
\cite{AlWe,St} for a detailed discussion).  M.~Gromov \cite{Gr}  found remarkable connections between the
polynomial (or the holomorphic disc) convexity of real manifolds and global rigidity of symplectic structures. 
In the present work we prove that a generic  Lagrangian surface in $\cx^2$ is polynomially convex near an 
isolated singularity  which is topologically an unfolded Whitney umbrella. This study is inspired by the work of 
A.~Givental \cite{G}, where he proved in particular that a compact real surface $S$ admits a smooth map 
$\iota: S \to \mathbb C^2$, isotropic with respect to the standard symplectic structure on $\cx^2$, such 
that the singularities of $\iota$ are isolated and either self-intersections or unfolded Whitney umbrellas. More 
precisely, if we denote by $z = x + iy$ and $w = u + iv$ the standard coordinates in $\cx^2$, then
$$\omega = dx \wedge dy + du \wedge dv$$ is the 
standard symplectic form on $\cx^2$. A smooth map $\phi: \cx^2 \to \cx^2$ is called {\it symplectic} if 
$\phi^*\omega = \omega$. Since such a map   is a local diffeomorphism, we call it  a (local) 
{\it  symplectomorphism}.  A smooth map $\iota: S\to (\cx^2, \omega)$ from a smooth real surface $S$ is 
called {\it isotropic} if $\iota^*\omega=0$.   A.~Givental~\cite{G} showed that near a generic point 
$p\in S$, which is an isolated singular point of $\iota$ of rank one, the map
\begin{equation}\label{e:whitney}
\pi: \mathbb R^2_{(t,s)} \to \mathbb R^4_{(x,u,y,v)}: (t,s) \to \left(ts,  \frac{2t^3}{3}, t^2, s\right)
\end{equation}
is a local normal form for $\iota$. In particular, this means that there exists a local 
symplectomorphism near $\iota(p)$ sending $\iota(S)$ onto a neighbourhood of the origin in 
$\Sigma := \pi(\rl^2)$.  The set $\Sigma$, as well as $\iota(S)$ near $\iota(p)$, is called the  
{\it unfolded (or open) Whitney umbrella}.  Our main result is the following.

\begin{theorem}\label{t:ga}
Suppose $\phi: \cx^2 \to \cx^2$ is either a generic real analytic symplectomorphism near the origin,
or the identity map. Then there exists a  neighbourhood of the point $\phi(0)$ in the surface $\phi(\Sigma)$ 
with compact polynomially convex closure.
\end{theorem}

The case where $\phi$ is the identity map is considered separately since it is
not generic. This implies that the Whitney umbrella $\Sigma$ is polynomially
convex near the origin. The above theorem also holds under weaker assumptions, namely, if $\phi$ is a generic 
local real analytic diffeomorphism and $D\phi(0)$, the differential of $\phi$ at zero, is symplectic, or if $\phi$ 
is a  $C^\infty$-smooth symplectomorphism with the jet  at the origin satisfying
some additional assumptions. See Section~5 for details.

Denote by  $\mathbb B(p,r)$ the open Euclidean ball of $\cx^2$ of radius $r> 0$ centred at $p$. As an 
application of Theorem~\ref{t:ga} we obtain the following result.

\begin{cor}\label{c:1}
Let $\phi$ be as in Theorem~\ref{t:ga}. Then for $\e>0$ sufficiently small, any continuous function on 
$\phi(\Sigma)\cap \mathbb B(\phi(0),\e)$ can be uniformly approximated by holomorphic polynomials.
\end{cor}

It will be shown in Section 4 that the genericity assumption of Theorem~\ref{t:ga} imposes restrictions only on 
the 2-jet of $\phi$ at the origin. More precisely, it suffices to require that such a jet does not lie in a real algebraic 
submanifold of codimension 2 (after the standard identification of the space of 2-jets at the origin with the Euclidean 
space). Our approach is based on the observation  that $\phi(\Sigma)$  is contained in the zero locus set $M$ of a 
strictly plurisubharmonic function with a unique critical point at the origin. Hence $M$ is a
 strictly pseudoconvex hypersurface smooth everywhere except the origin. This allows us to consider the characteristic 
 foliation induced on $\phi(\Sigma)$  by the embedding $\phi(\Sigma) \hookrightarrow M$. The origin is a unique 
 singular point for this foliation. It follows by the Hopf lemma that if $f$ is a holomorphic disc with boundary attached to 
 $\phi(\Sigma)$, then its boundary is transverse to the leaves of the characteristic foliation at every point different from 
 the origin.  Suppose now  that the  structure of  leaves of the characteristic foliation near the origin is topologically the 
 same as the phase portrait of a dynamical system near a saddle stationary point on the plane. Then the  boundary of 
 $f$ will {\it touch} a leaf of the characteristic foliation proving that such a holomorphic  disc does not exist. This 
 observation suggests a strategy for the proof of our main result.  The proof consists of two parts. 
 
First, we use Oka's Characterization Theorem for hulls \cite{O}, developed and adapted to the case under consideration 
in the work of G.~Stolzenberg \cite{S}, J.~Duval \cite{D} and B.~J\"oricke \cite{Jo1}. This enables us to generalize the 
above argument and prove polynomial convexity of $\phi(\Sigma)$ near the origin under the assumption that the phase 
portrait of the characteristic foliation is topologically a saddle (Sections 2 and 3). The remainder of the paper (Sections~4--7) 
is devoted to the study of the characteristic foliation near the origin. In Section 4 we write explicitly a 5-jet of the corresponding dynamical system on the plane; the origin is a stationary point with a high order of degeneracy. At the end of this section 
we describe explicitly the genericity assumption on the $2$-jet of $\phi$. Section 5 is expository: for the reader's 
convenience we recall relevant tools from the local theory of dynamical systems; in particular, we explain where the real 
analyticity assumption comes from. In Sections~6 and~7 we give a complete topological description of the phase portrait 
of the above dynamical system proving that it is a saddle. 

The problem remains open to determine local polynomial convexity for nongeneric Whitney umbrellas as we have no 
counterexamples to Theorem~\ref{t:ga}  if the genericity assumption is dropped. Our method relies on the properties 
of the phase portrait of the dynamical system associated with the characteristic foliation near the umbrella, and cannot be 
applied if some specific terms in the low-order jets at the origin of the map $\phi$ vanish. On the other hand, in applications 
to topological properties of surfaces the generic situation is often sufficient. Furthermore, our method works in some 
nongeneric cases, for instance, for the standard umbrella $\Sigma$ (this case is treated separately in Sections~4 and~6).

Convexity (polynomial, rational or holomorphic) of a Lagrangian or totally real manifold $E$ embedded into $\cx^n$  have 
been studied by several authors (see, for instance, \cite{Al, AlWe, Du1, DuGa, Gr, IvSh, St}). It is well known that the 
local polynomial convexity can fail near points where $E$ is not totally real. In the complex dimension $n = 2$, the 
tangent space of $E$ is a complex line, so such points are called {\it complex}; generically these points are 
isolated in $E$. The complex geometry of these points is well understood by now. There are three types  of generic 
complex points: elliptic, hyperbolic and parabolic (see, for instance, \cite{AlWe, St}), and the local polynomial convexity 
depends on the type. H.~Bishop \cite{Bi} and C.~Kenig - S.~Webster \cite{KW} proved that a neighbourhood of an 
elliptic point in $E$ has a nontrivial hull. On the other hand, F.~Forstneri\v c and E.~L.~Stout \cite{FoSt} proved that 
$E$ is locally polynomially convex near a hyperbolic point. The parabolic case is intermediate and in general both 
possibilities occur. This case was studied by B.~J\"oricke~\cite{Jo1,Jo2}. These results and their development have 
several important applications, in particular,  to the problem of complex and symplectic filling and topological classification 
of 3-contact structures.

In general, a compact real surface does not admit a Lagrangian or totally real embedding into 
$\cx^2$, for instance, torus is the only compact orientable real surface admitting  a Lagrangian embedding into 
$\cx^2$. By comparison, Givental's result is quite general as it applies to {\it all} compact surfaces. This makes it 
natural to study self-intersections and Whitney umbrellas on immersed Lagrangian manifolds in analogy with local
analysis of real surfaces near complex points.  Currently, only few results are obtained in this direction.  

The present work is the first step  in the study  of the most general case where  Whitney umbrellas arise. Our result 
implies that local convexity properties near a  generic real analytic  Lagrangian deformation of the standard  Whitney 
umbrella are similar to those of a hyperbolic point. This is a necessary step leading toward understanding of the
global geometry of immersed Lagrangian manifolds containing  Whitney umbrellas.

We thank S.~Nemirovski and V.~Shevchishin for bringing our attention to this problem and for helpful 
conversations. Also we would like to thank the anonymous referee for many constructive comments that
helped improve the exposition of the paper. The work on this paper was started in the fall of 2010 when the first author 
visited University Lille-I and the Laboratoire Paul Painl\'ev\'e, and was completed when the second author visited 
Indiana University and the University of Western Ontario in the fall of 2011. We thank these institutions for their support 
and excellent work conditions.

\section{Geometry of Whitney umbrellas}

The map $\pi: \rl^2_{(t,s)} \to \rl^4_{(x,u,y,v)}$ given by \eqref{e:whitney} is a smooth homeomorphism onto its 
image, nondegenerate except at the origin, where the rank of $\pi$ equals one. It satisfies $\pi^*\omega\equiv 0$, 
and so $\Sigma$ is a Lagrangian submanifold of $(\cx^2,\omega)$ with an isolated singular point at the origin. Thus,
$$
\Sigma = \{ (z,w)\in \cx^2: x=ts,\ u=\frac{2t^3}{3}, \ y=t^2,\ v=s;\ t,s\in \rl\}.
$$

The crucial role in our approach is played by  an auxiliary real
hypersurface $M$ defined by 
\begin{equation}\label{e:M}
M = \{(z,w)\in \cx^2: \rho(z,w)=x^2 - yv^2 + \frac{9}{4}u^2 -y^3=0\}.
\end{equation}
Clearly, $\Sigma$ is contained in  $ M$. Note that the hypersurface $M$ is smooth
away from the origin, and strictly pseudoconvex in $\mathbb B(0,\e)
\setminus \{0\}$ for $\e$ sufficiently small. 

Suppose now that $\phi: \cx^2 \to \cx^2$ is a local smooth diffeomorphism near the origin 
such that its linear part $D\phi(0)$ at the origin is a symplectic map.
Without loss of generality we may assume that $\phi(0)=0$.   The standard symplectic structure on 
$\cx^2$ is given by the matrix
$$\Omega = \left(\begin{array}{cc}0 & I_2\\ -I_2 & 0\end{array}\right),$$
where $I_2$ denotes the identity matrix on $\rl^2$. Similarly, we write
\begin{equation}
D\phi(0) = \left(\begin{array}{cc}A & B\\ C & D\end{array}\right).
\end{equation}
The condition that  $D\phi(0)$ is symplectic means that  $(D\phi(0))^t\, \Omega\, D\phi(0) = \Omega$ 
(where $t$ stands for matrix transposition). Therefore, the  real $(2
\times 2)$-matrices $A$, $B=(b_{jk})$, $C$, 
$D=(d_{jk})$ satisfy
\begin{equation}\label{e:symp_id}
A^t D - C^t B = I_2,\ A^tC = C^t A,  \ D^t B = B^t D.
\end{equation}

The standard complex structure of $\cx^2$ in real coordinates is given by the matrix
$$J=
\left( \begin{array}{cc}
0 & -I_2 \\
I_2 & 0  
\end{array}\right),
$$
which corresponds to multiplication by $i$. We perform an additional 
complex linear change of coordinates~$\psi$.
 Let $\psi: \rl^4 \to \rl^4$ be a linear transformation 
given by the $4\times 4$ matrix
\begin{equation}
\label{e:psi}
\left(\begin{array}{cc}D^t & -B^t\\ B^t & D^t\end{array}\right).
\end{equation}
This matrix commutes with $J$ and so  gives rise to a nondegenerate complex linear map in~$\cx^2$. 
Let $$\Sigma'=\psi\circ \phi(\Sigma),$$ and $$M'= (\psi\circ \phi)(M).$$ 

The differential at the origin of the 
composition $\psi \circ \phi$ is given by
\begin{equation}\label{e:differen}
D(\psi\circ\phi)(0)=
\left(\begin{array}{cc}I_2 & 0\\ E & G\end{array}\right),
\end{equation}
where we used identities \eqref{e:symp_id} to simplify the
matrix. Further, a direct calculation shows that 
\begin{equation}
\label{e:G}
G = ( g_{kj}) = B^t B + D^t D,
\end{equation}
and therefore, the matrix $G$ is symmetric with positive entries in the main 
diagonal. The determinant
\begin{equation}
\label{e:det}
\Delta = g_{11}g_{22} - g_{12}^2
\end{equation}
of $G$ coincides with that of the matrix in \eqref{e:psi} corresponding to a $\cx$-linear map of $\cx^2$. 
Hence  $\Delta$ is also positive. Let  $\rho' = \rho \circ (\psi \circ \phi)^{-1}$, and
\begin{equation}\label{e:omega}
\Omega' = \{(z',w')\in \cx^2 : \rho'(z',w') < 0\}.
\end{equation}
It follows from \eqref{e:M} and \eqref{e:differen} that 

\begin{equation}
\label{e:ro}
\rho'(z',w') = {x'}^2 + \frac{9}{4}{u'}^2 + o(\vert (z',w') \vert^2).
\end{equation} 
In particular,  the function $\rho'$ is 
strictly plurisubharmonic in a neighbourhood of the origin, and the hypersurface $M'$ is strictly pseudoconvex in 
a punctured neighbourhood of the origin.

\begin{lemma}\label{l:1}
The polynomial hull of the set $\mathbb B(0,\e)\cap \Sigma'$ for sufficiently small
$\e>0$ is contained in $\overline{\Omega' \cap \mathbb B(0,\e)}$.
\end{lemma}

\begin{proof} Choose $\e>0$  small enough such that $\rho'$ is strictly plurisubharmonic in $\mathbb B(0,\e)$. 
The polynomially convex hull of $\mathbb B(0,\e)\cap \Sigma'$ is contained in $\overline{\mathbb B(0,\e)}$. 
By a classical result (see, for instance, ~\cite{H}), the polynomially convex hull of 
$\overline{\mathbb B(0,\e)\cap \Sigma'}$ coincides with its hull with respect to the family of  functions 
plurisubharmonic in $\mathbb B(0,\e)$. Since for any point $p$ in 
$\overline{\mathbb B(0,\e)}\setminus \overline\Omega'$, we have $\rho'(p)>0$, the assertion of the lemma 
follows.
\end{proof}

\section{Characteristic foliation and polynomial convexity} 

In this section we explain the strategy of the proof of Theorem~\ref{t:ga}.

\subsection{Characteristic foliation.} Let $X$ be  a totally real surface embedded into  a real hypersurface $Y$ in $\cx^2$.  
Define on $X$ a field of lines determined at every  $p \in X $ by  
$$
L_p = T_p X \cap H_p Y,
$$ 
where $H_p Y = T_p Y \cap J(T_p Y)$ denotes the complex tangent line to $Y$ at the point $p$ and $J$ denotes the standard complex structure of $\cx^2$. Integral 
curves, i.e., curves which are tangent to $L_p$ at each point $p$, of this line field define a foliation on 
$X$. It is called the {\it characteristic foliation} of $X$.

We consider the characteristic foliation of $\Sigma\setminus\{0\} \subset M$ and 
$(\psi \circ \phi)(\Sigma)\setminus\{0\} \subset (\psi\circ\phi)(M)$.  
Characteristic foliations are invariant under biholomorphisms. Therefore, in order to study the characteristic foliation on $\phi(\Sigma)$ with respect to $\phi(M)$, it is sufficient to study the characteristic foliation 
of $\Sigma' = \psi \circ \phi (\Sigma)$ with respect to $M'$.

Recall that a {\it rectifiable arc} is a homeomorphic image of an interval under a Lipschitz map. 
Our ultimate goal is to prove the following. 

\begin{prop}\label{l:2}
There exist $\e > 0$ small enough and two rectifiable arcs $\gamma_1$ and $\gamma_2$ in  
$\Sigma' \cap \mathbb B (0,\e)$ passing through the origin with the following properties:
 \begin{itemize}
 \item[(i)]  $\gamma_j$ are smooth at all points except, possibly, the origin;
 \item[(ii)] $\gamma_1 \cap \gamma_2 = \{ 0 \}$;  
 \item[(iii)] if $K$ is a  compact subset of $\Sigma' \cap \mathbb B (0,\e)$ and is not contained in
$\gamma_1 \cup \gamma_2$, then there exists a leaf $\gamma$ of the characteristic 
foliation on $\Sigma'$ such that  $K\cap \gamma \ne \varnothing$ but $K$ does 
not meet both sides of $\gamma$.
\end{itemize}
\end{prop}
We point out that by (i) and (ii)  the union $\gamma_1 \cup \gamma_2$ does not bound any subdomain with the closure 
compactly contained in $\Sigma' \cap \mathbb B(0,\e)$.

The proof of the proposition will be given in Sections~\ref{s:reduction} - \ref{s:gen-hyp}.  Considering  pull-back of the 
characteristic foliation by $\psi \circ \phi \circ \pi$ we 
obtain a smooth vector field in a neighbourhood of the origin in $\rl^2_{(t,s)}$ with the stationary point at the origin.  
The study of its integral curves is based on the local theory of dynamical systems and can be read independently from the 
rest of the paper.

Assuming Proposition~\ref{l:2} we now prove our main results. The proof is based on  the  argument due to  
J.~Duval \cite{D} and B.~J\"oricke \cite{Jo1,Jo2}. Suppose that $\phi$ satisfies the assumptions of Theorem~\ref{t:ga}, 
and $\Sigma' = (\psi \circ\phi)(\Sigma)$. First we establish nonexistence of holomorphic discs attached to $\Sigma'$ near 
the Whitney umbrella. In what follows we denote by $\Delta$ the unit disc of  $\cx$. By a holomorphic disc we mean a  
map $f:\Delta \to \cx^2$ holomorphic in $\Delta$ and continuous on $\bar\Delta$. As usual, by its
boundary we mean the restriction $f\vert_{\partial\Delta}$; we identify it with its image $f(\partial \Delta)$. 

\begin{cor}
\label{disc-cor}
There exists $\delta > 0$ with the following property: a holomorphic disc 
$f:\Delta \to \mathbb B(0,\delta)$ with the boundary attached to $\Sigma'$, i.e., 
satisfying $f(\partial\Delta) \subset \Sigma'$, is constant.
\end{cor}

Before we proceed with the proof, we recall some basic notions. Let  $U \subset \rl^n$ be a domain and $N$ be a real 
submanifold of dimension $d$ in $U$. As usual, denote by   ${\mathcal D}(U)$  the space of  test-functions on $U$.   
The {\it current of integration} $[N]$ corresponding to $N$ is a continuous linear form on the space ${\mathcal D}^{d}(U)$ 
of differential forms of degree~$d$ with coefficients in ${\mathcal D}(U)$ defined by 
\begin{eqnarray}
\label{CI}
[N](\psi) = \int_N \psi, \, \, \, \forall \psi \in {\mathcal D}^d(U).
\end{eqnarray}
The current $[N]$ may be well-defined even when $N$ has some singularities provided that the behaviour of $N$ near its 
singular locus is not too bad. For instance, the current of integration over a complex analytic set or a rectifiable curve is well-defined, see \cite{Ch,Fe,Ha,St}. The exterior derivative $d[N]$ is then defined by duality: $d[N](\psi): = (-1)^{n-d+1}[N](d\psi)$. 

\begin{proof} Let $\e > 0$ be given by Proposition \ref{l:2}. Without loss of generality we may assume that~$\e$ is such 
that the function $\rho'$ in (\ref{e:ro}) is strictly plurisubharmonic in the ball $\mathbb B(0,2\e)$. Set $\delta = \e/2$. 
Suppose that there exists a nonconstant holomorphic disc $f: \Delta \to \mathbb B(0,\delta)$  with boundary glued to $\Sigma'$.     The function $\rho' \circ f$ is subharmonic in the unit disc, and so the maximum principle implies that $f(\Delta)$ is contained in 
$\Omega'=\{\rho' < 0 \}$. The proof consists of two parts. 

(1) First we 
show that the boundary of $f$ is not contained in $\gamma_1 \cup \gamma_2$. Arguing by contradiction, assume that 
$f(\partial\Delta) \subset \gamma_1 \cup \gamma_2$. The image $V:= f(\Delta)$ is a complex 1-dimensional 
analytic subset of $\Omega'$  and its boundary $bV:= \overline V \backslash V$  is contained in $\gamma_1 \cup \gamma_2$. 
Since the arcs $\gamma_j$ are rectifiable, it follows by the well-known results \cite{Ch,Ha,St}  that  two cases can occur.  
The first possibility is that the closure $\bar V$ is a complex $1$-dimensional analytic subset of $\cx^2$ contained in 
$\mathbb B(0,\e)$. This is impossible since a  closed complex analytic subset of positive dimension can not be compactly contained 
in $\cx^2$ (e.g.,~\cite{Ch}). The second case is when the area of $V$ is bounded, $V$ defines the current of integration $[V]$ on 
$\cx^2$,  and $d[V] = -[bV]$ in the sense of currents. Since $d^2 = 0$ for currents,  the current $[bV]$ is closed, i.e., 
$d[bV](\psi) = 0$ for all $\psi \in {\mathcal D}(\cx^2)$. Furthermore, there exists a closed subset $E$ in $bV$ of the Hausdorff 
1-measure  0  such that the couple $(V,bV)$ is a complex manifold with boundary in a neighbourhood of every point in $bV \setminus E$. Then $bV$ is the union of  closed subarcs  of the arcs $\gamma_j$. In particular, $bV$ is not a closed curve and has nonempty boundary in $\cx^2$. Let $p$ be a boundary point of $bV$ and  $U$  be   a sufficiently small neighbourhood of $p$ such that 
$U \cap bV$ is an arc in $U$ with the end $p$. Considering test-forms $\psi \in {\mathcal D}^1(U)$, we conclude  by Stokes' formula that $d[bV] \neq 0$ in $\cx^2$ since the Dirac mass at $p$ appears in the exterior derivative:  a contradiction.

(2) 
By the uniqueness theorem the set of points $f^{-1}(0)$ has measure zero on the unit 
circle. Since $\Sigma'$ is totally real outside the origin, it follows by the boundary regularity 
theorem~\cite{Ch} that $f$ is smooth (even real analytic) up to the boundary outside the
pull-back $f^{-1}(0)$. Applying the Hopf lemma (see, for instance, \cite {Ra}) to the 
subharmonic function $\rho' \circ f$ on $\Delta$ we conclude that $f$ is transverse to the
hypersurface $M'$ at every point different from the origin. Therefore, the
complex line tangent to $f(\Delta)$ at a boundary point is transverse to the
tangent complex line of $M'$ at this point. In particular,
the boundary $K: = f(\partial\Delta)$ is transverse to the leaves of the characteristic
foliation of $\Sigma'$. This contradicts Proposition~\ref{l:2}.
\end{proof}

\subsection{Sweeping out the envelope by analytic curves.} 
Given a compact set $K$, we denote by $\widehat K$ its polynomially convex hull. We also recall two useful related notions. 
The {\it essential hull} $K^{ess}$ of $K$ is defined by 
$$K^{ess} = \overline{\widehat K \setminus K},$$
and the {\it trace} $K^{tr}$ of $K^{ess}$ is  the intersection
$$K^{tr}= K^{ess}\cap K.$$
A local maximum principle of Rossi \cite{R,St} states that if $K$ is  a compact set in $\cx^n$, 
$E\subset \widehat K$ is compact, $U$ is an open subset of $\cx^n$ that contains $E$, and
if $f\in \mathcal O(U)$, then $||f||_E = ||f||_{(E\cap K)\cup \partial E}$, where the boundary 
of $E$ is taken with respect to $\widehat K$. By choosing $E=K^{ess}$ and 
$U=\cx^2$ we see that  ${K^{ess}}$ is contained in $\widehat{K^{tr}}$. Therefore, 
to prove that $K$ is polynomially convex, it is enough to show that $K^{tr}$ is empty. 

Let $$X=\Sigma' \cap \overline{\mathbb B(0,\e)}.$$
Then $X$ is a  closed disc, and the punctured disc $X \setminus \{ 0 \}$ is real analytically and total really embedded into  
$\partial \Omega' \setminus \{ 0 \}$, where $\Omega'$ is given by \eqref{e:omega}, and $\e$ is such that 
Lemma~\ref{l:1} holds.

\begin{prop}\label{duval}
The essential hull $X^{ess}$ cannot intersect a leaf of a characteristic foliation at a totally real point of $X$ without crossing it.
\end{prop} 

This result is due to J.~Duval \cite{D} (see also B.~J\"oricke \cite{Jo1}) in the case where a totally real disc is contained in 
the boundary of a smoothly bounded    strictly pseudoconvex domain of $\cx^2$. A detailed exposition of the proof is contained 
in \cite{St}. The proof, which is an application of Oka's method (developed also by G.Stolzenberg \cite{S}), is purely local 
and works without any essential modification in our case where $\partial \Omega'$ admits an isolated singularity at the origin. 
For reader's convenience we sketch the main steps of this construction. 

\medskip

{\it Step 1. Oka's Characterization Theorem.} We will state all results for dimension 2 because we deal with this case only; 
for more general versions see \cite{St, S}. 

 Let $U \subset O$ be two open subsets of $\cx^2$. Let $F: [0,1] \times U \to \cx$ be a continuous function that  
 for every $t \in [0,1]$ defines a nonconstant holomorphic function $f_t: = F(t,\bullet)$ on $U$. The zero locus 
 of  $f_t$, 
 $$
 V_t := \{ p \in U: f_t(p) = 0 \},\ \ t \in [0,1],
 $$ 
is a  purely 1-dimensional complex analytic subset of $U$. Suppose that every $V_t$ is also closed in $O$. Then  we call 
$V_t$  {\it an analytic curve in $O$} and call  $\{ V_t \}_{t \in [0,1]}$ a {\it continuous family of analytic curves} in $O$. 
The classical version of Oka's method is the following (see \cite{St}):

\medskip

\noindent{\bf Oka's Characterization Theorem.} {\it Let $K$ be a compact subset of $\cx^2$ and $O$ be a neighbourhood of 
$\widehat K$. If $\{ V_t \}$ is a continuous family of analytic curves in $O$ such that $V_0$ intersects 
$\widehat K$, but $V_1$ does not, then some $V_t$ must intersect $K$.}

\medskip

Many various versions of this fundamental principle are known. For us the following criterion is useful (cf. \cite{Du1}):
{\it Let $\{ V_t \}_{t \in [0,1]}$ be a continuous family of analytic curves in a neighbourhood $O$ of 
$\overline{\Omega' \cap \mathbb B(0,\e)}$  such that for all $t$ the curves $V_t$ do not intersect $X^{tr}$ and $V_1$ 
does not intersect $\overline \Omega'$. Then the curves $V_t$ do not intersect $X^{ess}$.}

Indeed, since the essential hull $X^{ess}$ is contained in $\widehat{X^{tr}}$ by Rossi's local maximum principle and $\widehat{X^{tr}}$ is contained in $\overline{\Omega' \cap \mathbb B(0,\e)}$ by Lemma~\ref{l:1}, it suffices to apply Oka's theorem.

The first step of the construction is the following key technical tool of \cite{D}:

\begin{lemma}\label{l:d} 
Let $p\in X\setminus \{0\}$ be an arbitrary point. Then $p$ does not lie in $X^{tr}$ if  there exist two
  continuous families $\{V_t\}_{t\in[0,1]}$ and $\{W_t\}_{t\in[0,1]}$
  of analytic curves in an open neighbourhood $O$ of $\overline{\Omega' \cap \mathbb B(0,\e)}$
with the following properties: 
\begin{itemize}
\item[(i)] $V_0$ and $W_0$ meet $X$ transversely at $p$ and with opposite signs of intersection;
\item[(ii)] for $t>0$, the varieties $V_t$ and $W_t$ are disjoint from $X^{tr}$; 
\item[(iii)] $V_1$ and $W_1$ do not intersect $\overline{\Omega'}$.
\end{itemize}
\end{lemma} 

Duval's original result is stated for the $\mathcal O(\overline G)$-hull of a smooth totally real surface $X\subset \partial
G$, where $G \subset \cx^2$ is a smoothly bounded strictly pseudoconvex domain. 
The proof is  also valid in our situation. Indeed, in order to show that $p$ does not belong to $X^{ess}$ it suffices to find a neighbourhood $U$ of $p$ such that $\widehat{X}$ does not intersect $U \setminus X$. Let $F,G:[0,1] \times O \to \cx$ 
be the functions  defining the families $\{V_t\}$, $\{W_t\}$ that satisfy conditions of the lemma. We use the notation 
$f_t = F(t,\bullet)$ and $g_t = G(t,\bullet)$. It follows from (i) that near $p$ the functions  $f_0$ and $g_0$ provide 
local holomorphic coordinates and the real surface $X$ is defined near $p$ by the equation $g_0 = h \circ f_0$. Here $h$ 
is a $C^2$-diffeomorphism in a neighbourhood of the origin in $\cx$, fixing the origin and reversing the orientation, so 
that $\vert h_{\overline\zeta}(0) \vert > \vert h_\zeta(0) \vert$. Denote by $\tau\Delta_-$ the left semidisc of radius $\tau > 0$,
that is, $\tau\Delta_- = \{ \zeta \in \cx: \vert \zeta \vert < \tau, \ \Re \zeta < 0 \}$. For $\alpha \in\tau\Delta_-$  and a complex parameter~$a$ consider the analytic curves $C_a$ in $O$ defined by the equation 
$$
(f_0 - a)(g_0 - h(a)) =  \alpha h_{\overline\zeta}(a).
$$
There exists $\tau > 0$ such that  when the parameter $a$ runs over a small neighbourhood of the origin in $\cx$ and $\alpha$ 
 runs over $\tau\Delta_-$, the family $\{ C_a \}$ fills out an open set $U \setminus X$ for a suitable 
neighbourhood $U$ of $p$. The proof  due to  \cite{D2}, Lemma 1, pp. 584-585,  is obtained by  the linear approximation of 
$h$ near $a$.  One verifies two properties of the family $C_a$.  First, given $\alpha \in \tau\Delta_-$ and $a$, the curve $C_a$ 
avoids $X$. Second, for every point $q \in U \setminus X$ one can find suitable $a$ and  $\alpha$ such that $C_a$ contains~$q$.  

Finally we note that every curve $C_a$ can be swept out of $\Omega'$ through a continuous family of analytic curves in $O$ 
in accordance with Oka's characterization of hulls. Such a  sweeping family of analytic curves is 
explicitly constructed in \cite{D} pp. 110-111, using the defining functions $f_t$, $g_t$ and the assumptions  
(ii) and (iii) of Lemma \ref{l:d}. 

This shows that no point near $p$ can be in $X^{ess}$, and therefore $p$ does 
not belong to $X^{tr}$. This verifies Lemma~\ref{l:d}.

\medskip

{\it Step 2: Construction of the families $\{V_t\}$ and $\{W_t\}$.} We employ the second part of the construction due to 
J.~Duval \cite{D}. 

Fix an orientation on  the real hypersurface $\partial\Omega'$ and the disc $X$. This allows one to define an orientation on the 
leaves of the characteristic foliation. Let $p \in X \setminus \{ 0 \}$ and $v_1$ and $v_2$ be vectors in the tangent space $T_pX$ giving a positively oriented basis there.
A nonzero vector $v$ tangent to the leaf of the characteristic foliation through $p$ defines the positive orientation on this leaf if the triple $v_1$, $v_2$, $Jv$ is a positively oriented basis of $T_p(\partial\Omega')$. Here $J$ denotes the standard complex structure of $\cx^2$, i.e., the vector $Jv$ can be identified with $iv$.

We argue by contradiction. Let $p \in X \setminus \{ 0 \}$ be a totally real point such that $p$ lies in the  leaf  $\gamma$ of the characteristic foliation, $p \in X^{ess}$, but $X^{ess}$ does not meet both sides of $\gamma$. Fix an open neighbourhood $U'$ 
of $p$ small enough so that $0$ does not lie in $\overline U'$ and $\Omega' \cap U'$ is biholomorphic to a strictly convex domain. More precisely, one can assume that there are local coordinates $(z',w')$ in $U'$ such that  $p$ corresponds to the origin $0'$, 
$U'$ is a ball and $\Omega' \cap U'$ is strictly convex. Let $x$ and $y$ be points on $X$ near $0'$ that lie on the same leaf of the characteristic foliation. Assume that the direction from $x$ to $y$ along this leaf is positive for the described above orientation. Denote by $L(x,y)$ the complex line through $x$ and $y$. Then $L(x,y)$ meets $X \cap U'$ at the points $x$ and $y$ only, this intersection is transversal, positive at $x$ and negative at $y$ , see \cite{D}, Lemma 2 . Denote by $\Delta(x,y)$ the intersection 
of the line $L(x,y)$ with the ball $U'$.

Denote by $\gamma'$ a leaf of the characteristic foliation near $p$ parallel to $\gamma$. By assumption, one can choose 
$\gamma'$ to be disjoint from $X^{ess}$ in $U'$. Consider a (short) arc $\alpha: [0,1] \to X \cap U'$ such that $\alpha(0) = p$, 
$\alpha(1) = p'$, where $p'$ is a point of $\gamma'$ and such that for $t > 0$ the point $\alpha(t)$ is  on the same side of 
$\gamma$ as the leaf $\gamma'$. Finally, choose a point $x \in \gamma$ which precedes $p$, and a corresponding  point $x' \in \gamma'$ which precedes $p'$ . Let  $\beta:[0,1] \to X$ be an arc in $\gamma'$ with  $\beta(0) = x'$, $\beta(1) = p'$. 

Now we are able to construct the first family $\{ V_t \}$ of analytic curves. We begin  with the family $\Delta(x',\alpha(t))$ where 
$0 \leq t \leq 1$. As it was mentioned above, the line $L(x,p)$ intersects $X$ with positive sign at $p$. This property is stable with respect to continuous deformations of  complex lines $L(q,p)$ where $q$ moves from $x$ to $x'$ in $X$. Hence, the first disc 
$V_0 = \Delta(x',\alpha(0))$ of our family  intersects $X$ at $p$ with positive sign. We continue this family  with the discs 
$\Delta(\beta(t),p')$ , $0 \leq t \leq 1$, starting with $t = 0$. When $t = 1$ we arrive to the complex tangent $\Delta(p',p')$. 
The final piece of the family $\{ V_t \}$ is obtained by the translation $\Delta(p',p')$ into the complement of $\Omega'$ along the outward normal direction to $\partial \Omega'$ at $p'$. Similarly, we proceed with the construction of the second family $\{ W_t \}$ 
using a point $y \in \gamma$ that succeeds $p$ along $\gamma$ and a corresponding point $y' \in \gamma'$ that succeeds $p'$ along $\gamma'$. 

The curves $V_0$ and $W_0$ meet transversally at $p$ with opposite signs of intersection and for $t > 0$ the curves $V_t$, $W_t$ do not meet $X^{tr}$.
In the above local coordinates $(z',w')$ on $U'$ these curves are intersections 
of the described above complex lines with  $U'$, i.e., the corresponding functions $f_t$, $g_t$ are degree one polynomials in $(z',w')$. Since the families $\{ V_t \} $ and $\{ W_t \}$ can be chosen  arbitrarily close to the complex
tangent line to $\partial \Omega'$ at $p$, their boundaries are contained in $\partial U'$ and do not intersect 
$\overline{\Omega'}$. Therefore  $V_t$ and $W_t$ are analytic curves in a suitably chosen global neighbourhood $O$ of 
$\overline{\Omega'\cap \mathbb B(0,\e)}$ in $\cx^2$.  Now  Step 1 can be used. Lemma \ref{l:d} implies that $p$ does not 
lie in $X^{ess}$, which gives a contradiction. Proposition \ref{duval} is proved.

\subsection{Proof of the main results.} We now prove the main results of the paper assuming that Proposition~\ref{l:2} holds.

\begin{proof}[Proof of Theorem~\ref{t:ga}.]
Let $\gamma_1$ and $\gamma_2$ be as in Proposition~\ref{l:2}. It follows from Propositions ~\ref{l:2} and~\ref{duval} that $X^{tr}$ is contained in the union $\gamma_1 \cup \gamma_2$, and Rossi's maximum principle
implies $X^{ess} \subset \widehat{\gamma_1 \cup \gamma_2}$.

A rectifiable arc is polynomially convex \cite{S}.  Moreover, if $Y$ is compact 
and polynomially convex, and $\Gamma$ is a compact connected set of finite length, 
then the set $(\widehat {Y\cup \Gamma}) \setminus (Y\cup \Gamma)$ is either empty or contains 
a complex purely 1-dimensional analytic subvariety of $\mathbb C^2\setminus (Y \cup \Gamma)$ 
(see \cite{St}, p.122). By taking $Y$ and $\Gamma$ to be our rectifiable curves $\gamma_j$, 
we see as in the proof of Corollary \ref{disc-cor} that their union cannot bound a complex  1-dimensional variety. 
Therefore, $\gamma_1 \cup \gamma_2$ is polynomially convex:  $\widehat{\gamma_1 \cup \gamma_2} = \gamma_1 \cup \gamma_2 \subset X$. As a consequence we obtain that 
$X^{ess}$ also is contained in $X$. Let $p$ be a point of $ \widehat X \setminus X$. Then $p \in X^{ess} \setminus X$ which is impossible. This implies that $\widehat X \setminus X$ is empty. 
Hence, $X$ is polynomially convex. Theorem~\ref{t:ga} is proved.
\end{proof}

\begin{proof}[Proof of Corollary~\ref{c:1}.] Let $\phi(0)=p$.
By Theorem~\ref{t:ga} there exists $\e>0$ such that  $X = \overline{\phi(\Sigma)\cap \mathbb B(p,\e)}$ is 
polynomially convex. We may further assume that $\phi(\Sigma) \cap \partial \mathbb B(p,\e)$ is a 
rectifiable  and even smooth curve. By the result of  J.~Anderson, A.~Izzo, and J.~Wermer \cite[Thm.~1.5]{AIW}, 
if $X$ is a polynomially convex compact subset of $\cx^n$, and $X_0$ is a 
compact subset of $X$ such that $X \setminus X_0$ is a totally real submanifold of $\cx^n$, 
of class $C^1$, then continuous functions on $X$ can be approximated by polynomials if and 
only if this can be done on $X_0$.  We apply this result to $X=\overline{\phi(\Sigma) \cap \mathbb B(p,\e)}$ 
and $X_0=\{p\}\cup (\phi(\Sigma) \cap \partial \mathbb B(p,\e))$. The set  $X_0$, is polynomially 
convex. Indeed, if not, we obtain as in the proof of Theorem \ref{t:ga} that $\widehat X_0 \setminus X_0$ contains  a complex purely 1-dimensional analytic subvariety $V$ of $\mathbb C^2\setminus X_0$. But then $V$ is contained in $\widehat X$, which contradicts Theorem \ref{t:ga}. Furthermore, by \cite{S2} or \cite{St}, p. 122,  continuous functions on 
$X_0$ can be approximated by polynomials. From this the corollary follows.
\end{proof}

The rest of the paper is devoted to the proof of Proposition \ref{l:2}.

\section{Reduction to a dynamical system}\label{s:reduction}

In this section we deduce  the dynamical systems  describing the pull-back in $\rl^2_{(t,s)}$ of the characteristic 
foliations on $\Sigma$ and $\Sigma'$. In Sections \ref{s:gen-hyp-s} and~\ref{s:gen-hyp} we will discuss the 
topological behaviour  of these foliations near the origin. For simplicity, the integral curves of these dynamical systems  
will also be called the leaves of the characteristic foliation. 

\subsection{Foliation on $\Sigma$.} The tangent plane to $\Sigma\setminus\{0\}$ is spanned by the 
vectors
\[  X_t = \left( \begin{array}{c}
s \\ 2t^2 \\ 2t \\ 0
\end{array}\right),\ \  X_s =
\left( \begin{array}{c}
t \\ 0 \\ 0 \\ 1
\end{array}\right).
\]
The directional vector of the characteristic line field is determined
from the equation
\begin{equation}\label{e:char-sys-s}
X= \alpha X_t + \beta X_s,
\end{equation}
where $\alpha= \alpha(t,s), \beta=\beta(t,s)$ are some smooth
functions on $\rl^2\setminus\{0\}$, and the vector $X$ belongs to the complex tangent $H_{\pi(t,s)}M$. Let  
$$
I_2 =
\left( \begin{array}{cc}
1 & 0 \\
0 & 1  
\end{array}\right),
\ \ 
J=
\left( \begin{array}{cc}
0 & -I_2 \\
I_2 & 0  
\end{array}\right).
$$
Multiplication by $i$ of a
vector in $\cx^2$ corresponds to multiplication by $J$ of the
corresponding vector in $\rl^4$. For $v\in T_p M$, the inclusion $v \in H_p M$ 
holds if and only if  $v, iv \in T_pM$. Therefore, 
$$
X \in H_{\pi(t,s)}M \ \Longleftrightarrow \ 
\langle J(\alpha X_t + \beta X_s) ,\nabla \rho \rangle = 0,
$$
where $\langle \cdot,\cdot \rangle$ is the standard Euclidean  product in $\rl^4$, and
$\nabla \rho$ is the gradient of the function $\rho$. Therefore, we
can choose
\begin{equation}\label{e:ab-s}
\alpha =  \langle JX_s, \nabla \rho \rangle, \ \  
\beta = -\langle JX_t, \nabla \rho \rangle.
\end{equation}
A calculation yields
$$
\nabla \rho = (2ts, 3t^3, -s^2 - 3t^4, -2t^2s),
$$
and
$$
\begin{array}{r}
\alpha = -3t^3-ts^2-3t^5,\\
\beta = s^3+4t^2 s + 7 st^4.
\end{array}
$$
Thus, 
\begin{equation}
X=\alpha X_t + \beta X_s = \alpha d\pi \left(\begin{array}{c}1\\0\end{array}\right)
+ \beta d\pi \left(\begin{array}{c}0\\1\end{array}\right)
= d\pi \left(\begin{array}{c}\alpha\\ \beta\end{array}\right),
\end{equation}
where $d\pi$ is the differential of the map $\pi$. It follows
that the characteristic foliation on $\Sigma\setminus\{0\}$ (or, more precisely, its pull-back on 
$\mathbb R^2 \backslash \{ 0 \}$ by the parametrization map $\pi$) is given by the system
of ODEs of the form 
\begin{equation}\label{e:hyp-s}
\left\{
\begin{array}{l}
\dot t = -3t^3-ts^2-3t^5\\
\dot s = s^3+4t^2 s + 7 st^4, 
\end{array}\right.
\end{equation}
where the dot denotes the derivative with respect to the time variable $\tau$.

\subsection{Foliation on $\Sigma'$.} Let $f: \rl^2 \to \rl^4$ be given by 
$$
f:= \psi \circ \phi \circ \pi,
$$ 
where we use the notation of the previous section. 
The directional vector of the characteristic foliation on $\Sigma'$ is determined by 
$$
X'=\alpha X'_t + \beta X'_s,
$$ 
where $X'_t = \partial f/\partial t$ and $X'_s = \partial f/\partial s$, and $\alpha= \alpha(t,s), \beta=\beta(t,s)$ 
are some smooth functions on $\rl^2\setminus\{0\}$ which are chosen in such a way that the vector $X'$ belongs 
to the complex tangent $H_{f(t,s)}M'$. We have
$$
X' \in H_{f(t,s)}M' \ \Longleftrightarrow \ 
\langle J(\alpha X'_t + \beta X'_s) ,\nabla \rho' \rangle = 0,
$$
where $\rho'$ is a defining function of 
$M'$, and the gradient $\nabla \rho'$ is expressed in terms of $(t,s)$ using the 
parametrization~$f$. Therefore, we can choose
\begin{equation}\label{e:ab}
\alpha(t,s) = \langle JX'_s, \nabla \rho' \rangle, \ \ \beta(t,s)=-\langle JX_t, \nabla \rho' \rangle.
\end{equation}
Thus, 
\begin{equation}
X'=\alpha X'_t + \beta X'_s = d f \left(\begin{array}{c}\alpha\\ \beta\end{array}\right).
\end{equation}
It follows that the characteristic foliation on $\Sigma'$ is determined by the system of ODEs of the form 
\begin{equation}\label{e:hyp}
\left\{
\begin{array}{l}
\dot t = \a(t,s)\\
\dot s = \beta(t,s) .
\end{array}\right.
\end{equation}

We write $f(t,s)=(f_1(t,s), \dots, f_4(t,s))$, where using \eqref{e:differen} 
and \eqref{e:whitney} we may express each $f_j$ as a power series in $(t,s)$:
\begin{multline}\label{e:f1}
f_1 (t,s) = x + \sum_{j+k+l+m \ge 2} \tilde f^1_{jklm}\, x^j u^k y^l v^m =\\
ts + f^1_{02} s^2 +  f^1_{12} ts^2 + f^1_{21} t^2s + f^1_{03}s^3 + \sum_{j+k\ge 4} f^1_{jk}t^j s^k,
\end{multline}
where $\tilde f^1_{jklm}$ and $f^1_{jk}$ are  real numbers. Similarly,
\begin{multline}\label{e:f2}
f_2 (t,s) = u + \sum_{j+k+l+m \ge 2} \tilde f^2_{jklm}\, x^j u^k y^l v^m =\\
\frac{2}{3}t^3 + f^2_{02} s^2 + f^2_{12} ts^2 + f^2_{21} t^2s + f^2_{03}s^3 + 
\sum_{j+k\ge 4} f^2_{jk}t^j s^k.
\end{multline}
Denote by $e_{jk}$ the entries of the matrix $E$ in \eqref{e:differen}.  Then
\begin{multline}\label{e:f3}
f_3 (t,s) = e_{11}x+e_{12}u +g_{11}y+g_{12}v + \sum_{j+k+l+m \ge 2} \tilde f^3_{jklm}\, x^j u^k y^l v^m =\\
g_{12}s + g_{11}t^2 + e_{11}ts + f^3_{02} s^2 + \frac{2e_{12}}{3}t^3  + f^3_{12} ts^2 + f^3_{21} t^2s + 
f^3_{03}s^3 + \sum_{j+k\ge 4} f^3_{jk}t^j s^k;
\end{multline}
\begin{multline}\label{e:f4}
f_4 (t,s) = e_{21}x+e_{22}u +g_{12}y+g_{22}v + \sum_{j+k+l+m \ge 2} \tilde f^4_{jklm}\, x^j u^k y^l v^m =\\
g_{22}s + g_{12}t^2 + e_{21}ts + f^4_{02} s^2 + \frac{2e_{22}}{3}t^3 + f^4_{12} ts^2 + f^4_{21} t^2s + 
f^4_{03}s^3 + \sum_{j+k\ge 4} f^4_{jk}t^j s^k.
\end{multline}
From these formulas we immediately obtain
\begin{equation}\label{e:x't}
X'_t = 
\left(\begin{array}{c}
s + 2 f^1_{21}ts + f^1_{12}s^2\\ 
2t^2 +2 f^2_{21}ts + f^2_{12}s^2 \\
2g_{11}t+e_{11}s+2e_{12}t^2+  2 f^3_{21}ts + f^3_{12}s^2  \\
2g_{12} t + e_{21} s + 2e_{22}t^2 + 2 f^4_{21}ts + f^4_{12}s^2 
\end{array}\right)
+ o(|(t,s)|^2),
\end{equation}
and 
\begin{equation}\label{e:x's}
X'_s = 
\left(\begin{array}{c}
t  + 2 f^1_{02}s + f^1_{21}t^2 + 2 f^1_{12}ts +3 f^1_{03}s^2 \\ 
2 f^2_{02}s + f^2_{21}t^2 + 2 f^2_{12}ts +3 f^2_{03}s^2 \\
g_{12}+e_{11}t +2 f^3_{02}s + f^3_{21}t^2 + 2 f^3_{12}ts +3 f^3_{03}s^2  \\
g_{22}+e_{21}t +2 f^4_{02}s + f^4_{21}t^2 + 2 f^4_{12}ts +3 f^4_{03}s^2  
\end{array}\right)
+ o(|(t,s)|^2).
\end{equation}

The defining equation of $M'$ can be chosen to be $\rho \circ (\psi \circ \phi)^{-1}$, where $\rho$
defines $M$ as in \eqref{e:M}. Let $(x',u', y', v')$ be the coordinates in the target domain of $\psi \circ \phi$,
in particular, we have $x'=f_1$, $u'=f_2$, $y'=f_3$, and $v'=f_4$. Let
\begin{equation}\label{e:inv-differen}
(D(\psi\circ\phi)(0))^{-1}=
\left(\begin{array}{cc}I_2 & 0\\ E' & G'\end{array}\right), \ E'=(e'_{jk}), \ G'=(g'_{jk}).
\end{equation}
Then
\begin{multline}
(\psi \circ \phi)^{-1}(x', u', y', v') = \left(
x' + \sum_{j+k+l+m \ge 2} h^1_{jklm}\, x'^j u'^k y'^l v'^m, \right. u' + \sum_{j+k+l+m \ge 2} h^2_{jklm}\, 
x'^j u'^k y'^l v'^m, \\
e'_{11}x'+ e'_{12}u'+g'_{11}y'+g'_{12}v' + \sum_{j+k+l+m \ge 2} h^3_{jklm}\, x'^j u'^k y'^l v'^m, \\
\left. e'_{21}x'+ e'_{22}u'+g'_{12}y'+g'_{22}v' + \sum_{j+k+l+m \ge 2} h^4_{jklm}\, x'^j u'^k y'^l v'^m \right).
\end{multline}
Therefore,
\begin{multline}\label{e:rho'}
\rho'(x', u', y', v') = \left(x' + \sum_{j+k+l+m \ge 2} h^1_{jklm}\, x'^j u'^k y'^l v'^m\right)^2 -\\
\left(e'_{11}x'+\dots+g'_{12}v' + \sum h^3_{jklm}\, x'^j u'^k y'^l v'^m\right) \cdot
\left(e'_{21}x'+\cdots+g'_{22}v' + \sum h^4_{jklm}\, x'^j u'^k y'^l v'^m\right)^2 \\
+ \frac{9}{4} \left(u' + \sum_{j+k+l+m \ge 2} h^2_{jklm}\, x'^j u'^k y'^l v'^m\right)^2 -
\left(e'_{11}x'+\dots+g'_{12}v' + \sum h^3_{jklm}\, x'^j u'^k y'^l v'^m\right)^3 
\end{multline}
Note that in \eqref{e:rho'} the only quadratic terms are $x'^2$ and $\frac{9}{4} u'^2$.
By taking partial derivatives in the above expression with respect to $x', u', y'$ and $v'$, and expressing
the resulting vector in terms of $(t,s)$ we will obtain the coordinates of the vector 
$$
\nabla \rho' = \left(\frac{\partial\rho'}{\partial x'},\frac{\partial\rho'}{\partial u'},\frac{\partial\rho'}{\partial y'},
\frac{\partial\rho'}{\partial v'}\right)=(R_x(t,s), R_u(t,s), R_y (t,s), R_v(t,s)).
$$ 
To determine the phase portrait of the characteristic foliation we will only need some low order terms in
the power series 
$$
\alpha(t,s) = \sum_{j,k \ge 0} \alpha_{jk} t^k s^j, \ \ 
\beta(t,s) = \sum_{j,k \ge 0} \beta_{jk} t^k s^j.
$$
Therefore, instead of explicit differentiation of \eqref{e:rho'},
we will employ a different strategy for computing coefficients of the terms of lower degree in the 
$(t,s)$-Taylor expansion of $\alpha$ and $\beta$.

\subsection{The power series of $\alpha$.}
We have 
\begin{equation}
\label{e:alpha}
\alpha(t,s) = \langle JX'_s, \nabla \rho' \rangle = -(X'_s)_3 \cdot R_x - (X'_s)_4 \cdot R_u  +
 (X'_s)_1 \cdot R_y + (X'_s)_2 \cdot R_v .
\end{equation}
We proceed in several steps computing the coefficients in the expansion for $\alpha$. To begin with, there cannot 
be a free term in the power series of $\alpha$ because every term in $\nabla \rho'$ will necessarily have positive 
degree in $t$ or $s$. 

\smallskip

\noindent{\it Term $t$}: Since no component of $\nabla \rho'$ can contain a degree zero term or the monomial $t$, 
there is no term $t$ in $\alpha$.

\smallskip

\noindent{\it Term $s$}: The first two components of $X'_s$ do not contain free terms, therefore, monomial $s$
can appear in $\alpha$ only if $R_x$ or $R_u$ will contain it. By inspection of \eqref{e:f1} - \eqref{e:f4} we see
that $y'$ and $v'$ are the only terms that can produce monomial $s$. Therefore, for $s$ to appear in $R_x$ or $R_u$,
the function $\rho'$ must contain at least one of the terms $x' y'$, $x' v'$, $u' y'$ or $u' v'$. However, from 
\eqref{e:rho'} neither of these terms exists. Thus, there is no monomial $s$ in the power series of $\alpha$.

\smallskip

\noindent{\it Term $ts$}: We inspect terms in $X'_s$ of degree lower than $ts$. These appear in $(X'_s)_1$
(terms $t$ and $s$), in $(X'_s)_2$ (term $s$), in $(X'_s)_3$ (a free term, $t$ and $s$), and in $(X'_s)_4$ 
(a free term, $t$ and $s$). Therefore, for $ts$ to appear in $\alpha$, at least one of the following options must occur:
\begin{enumerate}
\item either $R_x$ or $R_u$ has $t$, $s$ or $ts$;
\item $R_y$ has either $t$ or $s$
\item $R_v$ has $t$. 
\end{enumerate}
Of the above three options only (1) can happen: $\rho'$ contains the term $x'^2$, and therefore, $R_x$
contains $2ts$. It follows now from \eqref{e:f1},\eqref{e:x's}  and \eqref{e:alpha} that $\alpha_{11}=-2 g_{12}$.

To simplify further considerations, we note that term $t$ cannot occur in any of the components of
the vector~$\nabla\rho'$.

\smallskip

\noindent{\it Term $t^2$}: By inspection of $X'_s$, we conclude that either $R_x$ or $R_u$ has term 
$t^2$, so $\nabla\rho'$ must have either $x'y'$, $x'v'$, $u'y'$ or $u'v'$, neither of which appears. This
means that $\alpha$ does not contain term~$t^2$.

\smallskip

\noindent{\it Term $s^2$}: By inspection of $X'_s$, the following options  are possible:
\begin{enumerate}
\item either $R_x$ or $R_u$ has $s$ or $s^2$;
\item either $R_y$ or $R_v$ has term $s$.
\end{enumerate}
Option (2) is impossible, but $\rho'$ can have terms $u'^2$, $u'v'^2$ or $u'y'^2$ which gives (1). We have 
the following   expression for $\alpha_{02}$, which depends on the coefficients of the Taylor
expansion for $(\psi \circ \phi)^{-1}$:
$$\alpha_{02} = \frac{9}{2}(h^2_{0002}g^2_{22} + f^2_{02} + h^2_{0020}g^2_{12}).$$

\smallskip

\noindent{\it Term $t^3$}: By inspection of $X'_s$,  the following options are possible:
\begin{enumerate}
\item either $R_x$ or $R_u$ has at  least one of $t^2$ or $t^3$;
\item $R_y$ has $t^2$.
\end{enumerate}
Option (2) can happen only if $\rho'$ would have $y'^2$ or $y'v'$, which is impossible. For the
same reason in option (1) terms $R_x$ or $R_u$ cannot produce $t^2$. The only term in $\nabla\rho'$ 
that can produce $t^3$ is $u'$. Therefore, the only possibility in (1) is the term $t^3$ in $R_u$, 
which indeed happens since $\rho'$ contains $u'^2$. It follows that $\alpha_{30}=-3g_{22}$.

\smallskip

Thus,
$$
\alpha(t,s) = -2g_{12}ts + \alpha_{02}s^2 - 3g_{22}t^3 + \sum_{j+k>2,\ (j,k)\ne (3,0)} \alpha_{jk} t^j s^k.
$$

\subsection{The power series of $\beta$.}
We have
$$
\beta(t,s) = -\langle JX'_t, \nabla \rho' \rangle = (X'_t)_3 \cdot R_x +(X'_t)_4 \cdot R_u  - (X'_t)_1 \cdot R_y - 
(X'_t)_2 \cdot R_v .
$$
Again, there cannot be a free term in $\beta$ because every term in $\nabla \rho'$ will necessarily have positive
degree in $t$ or $s$. Further, no component in $\nabla\rho'$ can produce a term $t$, and so the power series of
$\beta$ cannot contain monomial~$t$.

\smallskip

\noindent{\it Term $s$}: Since no component of $X'_t$ contains a free term, $\beta$ cannot have 
monomial~$s$.

\smallskip

\noindent{\it Term $ts$}: By inspection of $X'_t$ we conclude that either $R_x$ or $R_u$ must
have term $s$, which is impossible. Hence, $\beta$ does not contain monomial $ts$.

\smallskip

\noindent{\it Terms $t^2$ and $s^2$}: Analogous considerations show that these terms cannot appear
in $\beta$.

\smallskip

\noindent{\it Term $t^2s$}: By inspection of $X'_t$ the following is possible for $R$:
\begin{enumerate}
\item $R_x$ has at  least one of $t^2$, $s$, or $ts$;
\item $R_u$ has at  least one of $t^2$, $s$, or $ts$;
\item $R_y$ has $t^2$;
\item $R_v$ has $s$.
\end{enumerate}
Options (3) and (4) imply that $\rho'$ has $v'^2$, $y'^2$, or $v'y'$, neither of which is possible.
Option (2) implies that $\rho'$ has $u'y'$, $u'v'$ and $u'x'$. Neither of these terms are present
in $\rho'$, so (2) is also not possible. Option (1) implies that $\rho'$ has at least one of $x'y'$, $x'v'$,
or $x'^2$. Only the latter happens, and so $\beta_{21} = 4g_{11}$.

\smallskip

\noindent{\it Term $ts^2$}: This term can appear in $\beta$. We have
$$\beta_{12} = 2e_{11} + 6 g_{12}f^2_{02}.$$

\smallskip

\noindent{\it Term $t^3$}: By inspection of $X'_t$, the only option is that either $R_x$ or
$R_u$ has term $t^2$. This is however not possible.

\smallskip

\noindent{\it Term $t^4$}: The possibilities for $R$ are as follows:
\begin{enumerate}
\item $R_x$ has at  least one of $t^2$ or $t^3$;
\item $R_u$ has at  least one of $t^2$, or $t^3$;
\item $R_v$ has $t^2$.
\end{enumerate}
Option (3) cannot occur. The only possible option in (1) or (2) is that $t^3$ appears in $R_u$. This 
comes from the term $u'^2$ in $\rho'$. It follows that
$\beta_{04}=6g_{12}$.

\smallskip

\noindent{\it Term $s^3$}:  We have
$$\beta_{03} = 2e_{11}f^1_{02} + \frac{9}{2}e_{21}f^2_{02} .$$

\smallskip

Combining everything together we get
$$
\beta(t,s) = 4g_{11}t^2s + \beta_{12}ts^2 + \beta_{03}s^3+ 6g_{12}t^4 + 
\sum_{j+k>3,\ (j,k)\ne (4,0)} \beta_{jk} t^j s^k.
$$

\bigskip

We note that if $\phi$ is merely a smooth diffeomorphism, then the above calculations give the 
values for the jets of $\alpha$ and $\beta$ at the origin of the corresponding orders. In either case
the characteristic foliation on $\Sigma'$ is given by
\begin{equation}\label{e:system}
\left\{
\begin{array}{l}
\dot t = \alpha(t,s) = -2g_{12}ts + \alpha_{02}s^2 - 3g_{22}t^3 + o(|t|^3+|s|^2+|ts|)\\
\dot s = \beta(t,s) = 4g_{11}t^2s + \beta_{12}ts^2 + \beta_{03}s^3+ 6g_{12}t^4 + 
o(|t^2s| + |ts^2| + |s|^3+ |t|^4) .
\end{array}
\right.
\end{equation}

It is easy to see that for a generic symplectomorphism $\phi:(x,u,y,v) \mapsto (x',u',y',v')$  and a generic $\psi$ 
the coefficients $\alpha_{02}$, $\beta_{12}$, $\beta_{03}$ do not vanish. Indeed, if $\psi$ is close to the identity 
map and the component $u'$ of $\phi$ contains the term $a v^2$ with $a \neq 0$, then $f^2_{02} \neq 0$ and  
$\alpha_{02}$, $\beta_{12}$, $\beta_{03}$ do not vanish. Therefore, they do not vanish generically.

\medskip

{\bf Remark.} It follows from the above considerations that our restriction  on $\phi$ to be generic involves only 
the 2-jet of $\phi$ at the origin. In other words, it suffices to require in Theorem \ref{t:ga} that $\phi$ has a generic 
2-jet at the origin.

\begin{lemma}\label{l:iso}
Let $\phi$ be a local symplectomorphism near the origin, and let $\mathcal X$ be the vector field near the origin
in $\rl^2$ corresponding to the characteristic foliation on $\Sigma'$. Then $\mathcal X$ does not 
vanish outside the origin.
\end{lemma}

\begin{proof}
Since $\phi$ is symplectic, $\phi (\Sigma \setminus \{0\})$ is a Lagrangian surface, in particular, totally real. 
Therefore, $\psi \circ \phi (\Sigma\setminus \{0\})$ does not contain complex points. Further, it easily follows 
from \eqref{e:ab} that $\alpha(t_0,s_0)=\beta(t_0,s_0)=0$, $(t_0,s_0) \ne 0$ if and only if 
$f(t_0, s_0)$ is a complex point of $\Sigma'$. From this the result follows.
\end{proof}

\section{Generalities on planar vector fields}
For the proof of Proposition~\ref{l:2} we need to determine the topological structure of the
{\it orbits} or {\it maximal integral curves} associated with the vector fields defined by~\eqref{e:hyp-s} 
and~\eqref{e:system}. Both systems have higher order degeneracy (the linear part vanishes) at the origin, 
and consequently it is a {\it nonelementary} singularity of~\eqref{e:hyp-s} and~\eqref{e:system}. 
Therefore, standard results, such as the Hartman-Grobman theorem, do not apply here. Instead, 
we will use some more advanced tools from dynamical systems. We will be primarily interested in 
understanding the topological picture of~\eqref{e:hyp-s} and~\eqref{e:system} near the origin up to a 
homeomorphism preserving the orbits. In this section we outline relevant results and recall some common 
terminology.

\subsection{Finite jet determination of the phase portrait.}
The local phase portrait of a vector field near a nonelementary isolated singularity can be determined through a 
finite sectorial decomposition. This means that a neighbourhood of the singularity is divided into a finite number
of sectors with certain orbit behaviour in each sector. If the vector field has at least one {\it characteristic}
orbit (i.e., orbits approaching in positive or negative time the singularity with a well-defined slope limit), then the 
boundaries of the sectors can be chosen to be characteristic orbits. The overall portrait is then understood by 
gluing together the topological picture in each sector. The general result due to Dumortier~\cite{d} 
(see also~\cite{DLA}) can be stated as follows: 

\smallskip

{\it Suppose that a $C^\infty$-smooth vector field $\mathcal X$ singular at the origin in $\rl^2$ satisfies  
the \L ojasiewicz inequality 
$$
|\mathcal X(x)| \ge c |x|^k, \ \ c>0, \ k\in \mathbb N,
$$
for $x\in \rl^2$ is some neighbourhood of the origin. Then $\mathcal X$ has the finite sectorial decomposition 
property, that is, the origin is either a centre (all orbits are periodic), a focus/node (all orbits terminate at
the origin in positive or negative time), or there exists a finite number of characteristic orbits which bound
sectors with a well-defined orbit behaviour (hyperbolic, parabolic, or elliptic). If the vector field $\mathcal X$
has a characteristic orbit, then its phase portrait is determined by its jet of finite order $k$, in the sense that
any other vector field with the same jet of order $k$ at the origin has the phase portrait homeomorphic to
that of $\mathcal X$. Further, whether the vector field $\mathcal X$ has a characteristic orbit depends 
only on a jet of $\mathcal X$ of some finite order.}

\smallskip

The original proof of the above result in \cite{d} is based on the desingularization by 
means  of successive (homogeneous) blow-ups. After each blow-up the singularity is replaced by a circle, 
and after a finite number of such blow-ups one obtains a vector field with only nondegenerate singularities. 
The construction of the blow-up maps depends only on a finite order jet of the original vector field at the 
origin. From the configuration of the singularities of the modified system on the preimage of the origin under 
the composition of blow-ups, it is always possible to deduce if the original vector field has a characteristic
orbit. If such an orbit exists, then the singularity is not a centre or a focus, and the phase portrait is 
determined by a jet of finite order. Further, the \L ojasiewicz inequality holds for any real analytic vector field 
in a neighbourhood of an isolated singularity (see, e.g., \cite{BM}) and, in particular, in our case, in 
view of Lemma~\ref{l:iso}.

Alternatively, it is possible to use {\it quasihomogeneous} blow-ups, which are chosen according to the Newton 
diagram associated with $\mathcal X$ (see \cite{P}). The advantage is that this gives a computational algorithm 
for constructing the sectorial decomposition for a particular system. A detailed discussion of this approach for
real analytic systems is given in Bruno \cite{B} in the language of normal forms. Using Bruno's method we will
show that for a real analytic $\phi$ in general position, the vector field defined by \eqref{e:system} will 
always have a characteristic orbit, and its phase portrait near the origin is a saddle. 

If in Theorem~\ref{t:ga} the map $\phi$ is smooth, then the vector field corresponding to the characteristic 
foliation is only smooth, and the \L ojasiewicz inequality imposes additional assumption on the vector field,
and therefore on $\phi$. The \L ojasiewicz condition depends on the jet of the vector field at the origin and holds 
for all jets outside a set of infinite codimension in the space of jets, but it is not clear whether for a generic 
smooth symplectomorphism the inequality is satisfied. However, assuming that the \L ojasiewicz condition 
does hold, the topological picture of the characteristic foliation is determined by its finite jet at the origin. 
Therefore, we may consider a polynomial
vector field obtained by truncation of \eqref{e:system} at sufficiently high order without distorting the phase
portrait of the system. After that we may apply Bruno's method to determine its geometry. Thus, in 
Theorem~\ref{t:ga} we may assume that $\phi$ is a generic smooth symplectomorphism such that the vector 
field corresponding to the characteristic foliation satisfies the \L ojasiewicz inequality.
 
If in Theorem~\ref{t:ga} the map $\phi$ is a real analytic diffeomorphism with $D\phi(0)$ symplectic, then all 
of the arguments go through provided that the vector field \eqref{e:system} vanishes at the origin only. 
The latter holds for the following reason: consider near the origin the complexification $F$ of the real analytic map 
$f=\psi \circ \phi \circ \pi: \rl^2 \to \cx^2$. Then $F: \cx^2 \to \cx^2$ is a holomorphic map such that 
$F|_{\rl^2_{(t,s)}}=f$, in particular, $F(\rl^2)=\Sigma'$. Moreover, since $f$ has rank 2 outside the origin,
it follows that the Jacobian of $F$ does not vanish on $\rl^2\setminus\{0\}$, and therefore, $F$ is a local
biholomorphism near any point on $\rl^2\setminus\{0\}$. But this implies that $\Sigma'\setminus \{0\}$
is totally real, and therefore the characteristic foliation has no singularities outside the origin. Thus,
Theorem~\ref{t:ga} holds under the assumption that $\phi$ is a generic real analytic diffeomorphism with
$D\phi(0)$ symplectic.

\medskip

In the remaining part of this section we outline general theory of normal forms and sector decomposition of dynamical systems 
due to Bruno~\cite{B}, while the actual numerical calculations for~\eqref{e:hyp-s} and~\eqref{e:system} are presented in
Section~\ref{s:gen-hyp-s}.

\subsection{Normal forms for elementary singularities.} We state three theorems due to Bruno
on normal forms for vector fields near an isolated {\it elementary} singularity. Consider the 
system

\begin{equation}\label{e.Sys}
\dot x_i = \lambda_i x_i + \sigma_i x_{i-1} + \varphi_i(X), \  i = 1, 2,
\end{equation}
where $x_i$ are smooth functions of a real variable and  $X=(x_1,x_2)$. Here  $\sigma_j, \lambda_j$  are real, $\sigma_1 = 0$  and the series $\varphi_i$ does not contain constant or linear terms. In other words, using the notation $X^Q = x_1^{q_1} x_2^{q_2}$ for  $Q=(q_1, q_2)\in \mathbb Z^2$,  we can write 
\begin{equation*}
\varphi_i (X)= \sum_{Q} f_{iQ} X^Q, \  i = 1, 2,
\end{equation*}
where $q_j \ge 0$, $q_1 + q_2 >0$. The main assumption is that at least one of the eigenvalues $\lambda_i$ is nonzero that is $|\lambda_1|+|\lambda_2| \ne 0$. This means that the  origin is an elementary singularity. We suppose below that all systems considered in the Normal Forms Theorems are real analytic, though the considerations in the formal power series category also make sense.

The goal is to transform system (\ref{e.Sys}) to the simplest possible form 

\begin{equation}
\label{NF1}
\dot y_i = \tilde \psi_i(Y):= \lambda_i y_i + \sigma_iy_{i-1} + \psi_i(Y), \ i = 1, 2
\end{equation}
by a local invertible 
change of coordinates 
\begin{equation}\label{change}
x_i = y_i + \xi_i(Y), \ i = 1, 2,
\end{equation}
where the series $\xi_i$ in $Y = (y_1,y_2)$ do not contain constant or linear terms:
$$\xi_i(Y) = \sum_{|Q|>1} h_{iQ} Y^Q, \ i =1,2.$$
Here and below we use the notation $\vert Q \vert = \vert q_1 \vert + \vert q_2 \vert$.
Such a change of coordinates in general is not real analytic, i.e., the series $\xi_i$ can be divergent. 
For this reason we consider formal power series $\xi_i$ and refer to~\eqref{change} as a {\it formal} changes of coordinates.

It is convenient to use the representation
\begin{equation}
\label{NF2}
\tilde \psi_i(Y) = y_ig_i(Y) = y_i \sum_{Q \in N_i} g_{iQ} Y^Q, i = 1,2,
\end{equation}
where
$$N_1 = \{Q = (q_1,q_2) \in  \mathbb Z^2: q_1 \ge -1, q_2\ge 0 , q_1+q_2 \ge 0\},$$ 
$$N_2 = \{Q = (q_1,q_2) \in  \mathbb Z^2:  q_1\ge 0 , q_2 \ge -1,  q_1+q_2 \ge 0\}.$$ 

Set  $\Lambda = (\lambda_1,\lambda_2)$ and denote by $\langle\bullet,\bullet\rangle$  the standard inner product in $\rl^2$.

\medskip {\it 

\noindent {\bf Principal Normal Form} \cite[Ch. II, \S 1, Thm 2, p. 105]{B}: There exists a formal change of coordinates (\ref{change}) 
such that system~\eqref{e.Sys} in the new coordinates takes the form (\ref{NF1}) 
where $g_{iQ}=0$ for $Q = (q_1,q_2)$ satisfying 
$\langle Q, \Lambda \rangle = q_1 \lambda_1 + q_2\lambda_2 \neq 0$.}

\medskip

Therefore the normal form  (\ref{NF1}) contains only terms of the form $y_ig_{iQ}Y^Q$ satisfying
\begin{equation}
\label{res}
\langle Q,\Lambda \rangle  = 0 .
\end{equation}
Such terms are called {\it resonant}.

The fundamental question on the convergence of a normalizing change of coordinates for an analytic system ~\eqref{e.Sys} is discussed in \cite{B}. In the cases which we will consider below, normalizing changes of coordinates (\ref{change}) will be 
analytic or at least $C^\infty$-smooth local diffeomorphisms (see \cite{B}). This is sufficient for the study of local topological behaviour of integral curves.

\medskip

Consider now a more general system of two differential equations in two variables of the form 
\begin{equation}\label{e:br0}
\dot x_i = \lambda_i x_i + x_i \sum_{Q\in \bf V} f_{iQ} X^Q=\lambda_i x_i + x_i f_i, \ \ i=1,2,   
\end{equation}
where $\Lambda = (\lambda_1, \lambda_2) \ne 0$. The set  $\bf V \subset \mathbb Z^2$, over which the exponents $Q$ run, 
is to be prescribed. In the hypothesis of the Principal Normal Form Theorem, $\varphi_i(X)$ are power series in nonnegative 
powers of variables and the corresponding $\bf V$ is almost completely contained in the first quadrant of the plane. 

To formulate a weaker assumption on $\bf V$  we consider two vectors $R^*$ and $R_*$ in 
$\mathbb R^2$ contained in the second and the forth quadrant respectively, and denote by  $\bf V$ the sector bounded by $R^*$ and $R_*$ and containing the first quadrant. We assume that $R^*$ and $R_*$ are such that
$\bf V$ has angle less than $\pi$. As a consequence, the sector $\bf V$ is the convex cone generated by $R^*$ and $R_*$ i.e. consists of the vectors $\alpha_1 R^* + \alpha_2 R_*$ with  $\alpha_j \geq 0$.  We use the notation $\vert X \vert = (\vert x_1 \vert, \vert x_2 \vert)$ and $\vert X \vert^Q = \vert x_1 \vert^{q_1}\vert x_2 \vert^{q_2}$. 

Denote by $\mathcal V(X)$ the space of power series $\sum_Q f_Q X^Q$, where $Q\in \bf V$. Since in our situation such a 
series can have an infinite number of terms with negative exponents (even after multiplication by $x_i$), the notion of its convergence requires clarification. Consider first a numerical series
\begin{equation}
\label{ser1}
\sum_{Q \in {\mathbb Z^2}} a_Q
\end{equation}
where the indices $Q$ run through ${\mathbb Z}^2$. Let $(\Omega_n)$ be an increasing exhausting sequence of bounded 
domains in $\rl^2$. Set 
$$S_n = \sum_{Q \in \Omega_n} a_Q$$
(the partial sums). If the sequence $(S_n)$ admits the limit $S$ and this limit is independent of the choice of the sequence 
$(\Omega_n)$, then we say that series  (\ref{ser1}) converges to the sum $S$. It is well-known that if for some sequence 
$(\Omega_n)$ the sequence of the partial sums of the series 
\begin{equation}
\label{ser2}
\sum_{Q \in {\mathbb Z}^2} \vert a_{Q}\vert
\end{equation}
converges, then series (\ref{ser1}) and (\ref{ser2}) converge. In this case we say that series (\ref{ser1}) converges absolutely. 

Under the above assumptions on $R^*$ and $R_*$   a series of class $\mathcal V(X)$ is called {\it convergent}  if it converges absolutely in the set
\begin{equation}\label{e.sss}
\mathcal U_{\bf V}(\e) = \left\{ X : |X|^{R_*} \le \e, |X|^{R^*} \le \e, |x_1| \le \e, |x_2| \le \e \right\},
\end{equation}
for some $\e > 0$. As explained in detail in \cite{B}, this subset of the real plane is a natural domain of convergence for such a series. As an example we notice that when the sector $\bf V$ is defined by the vectors $R_*=(1,0)$ and $R^*=(0,1)$,  i.e., 
coincides with the first quadrant,  then  the class ${\mathcal V}(X)$ coincides with the class of usual power series with nonnegative exponents and the set $\mathcal U_{\bf V}(\e)$  coincides with the bidisc of radius $\e$. 

Let $\bf V$ be a sector which determines system (\ref{e:br0}). We consider  changes of variables of the form

\begin{equation}\label{e:br1}
x_i = y_i + y_ih_i(Y),  \ \ i=1,2,
\end{equation}
where $h_i \in {\mathcal V}(Y)$, i.e., $h_i(Y) = \sum_{Q\in \bf V} h_{iQ} Y^Q$. In the new coordinates the system 
takes the form 
\begin{equation}\label{e:br2}
y_i = \lambda_i y_i + y_i g_i(Y), \ \ i=1,2.
\end{equation}

\medskip {\it

\noindent {\bf Second Normal Form} \cite[Ch. II, \S 2, Thm 1, p. 128]{B}: Suppose that $\bf V$ is a sector as described above. Then system \eqref{e:br0} can be transformed by a formal change of variables~(\ref{e:br1})  into a normal form (\ref{e:br2})  
with $g_i \in {\mathcal V}(Y)$. The coefficients of $g_i$ satisfy $g_{iQ}=0$ if $\langle Q, \Lambda \rangle \ne 0$. 
}

\medskip

The normalizing change of coordinates in the above theorem in general is not convergent, even if system \eqref{e:br0} is analytic. However, such a change of coordinates is always convergent or $C^\infty$-smooth in $\mathcal U_{\bf V}(\e)$.  For this reason the behaviour of the integral curves of systems~\eqref{e:br0}  and~\eqref{e:br2} coincide in the sector given by~\eqref{e.sss} 
for sufficiently small $\e>0$.

\medskip

The third theorem deals with the case somewhat intermediate with respect to the two previous theorems.
Let $\bf V$ be the sector in \eqref{e:br0} defined as above by the vectors $R^*$ and $R_*$. Assume that 
$R^*=(r^*_1, r^*_2)$,  $R_*=({r_1}_*, -1)$ with $r^*_1 < 0 < r^*_2$, ${r_1}_*>0$, and 
$|r^*_1/r^*_2|<{r_1}_*$. Note that the conditions on $r^*_1$, $r^*_2$, and ${r_1}_*$ exactly mean 
that $R^*$ and $R_*$ are in the second and forth quadrants respectively and the angle of ${\bf V}$ is less 
than $\pi$. 

The  additional  assumption which we impose is  that the  expressions on the right-hand side of~\eqref{e:br0} are the series 
in integer nonnegative powers of $x_2$. Since the series $f_1(X)$ does not contain negative powers of $x_2$, the coefficient  
$f_{1Q}$ in $f_1(X)$ vanishes unless the vector $Q$ lies in the sector
$$
_1{\bf V} = \left\{Q: Q=\alpha_1 R^* + \alpha_2 \cdot (1,0), \  \alpha_1, \alpha_2 \ge 0 \right\}.
$$

Denote by $_1\mathcal V(X)$ the class of such series $f_1$. Furthermore, since $x_2f_2(X)$ also does not contain 
negative powers of $x_2$, the coefficient $f_{2Q}$ in $f_2(X)$ of \eqref{e:br0}
will vanish unless the vector $Q$ lies either in $_1{\bf V}$, or along the ray $\{q_2=-1$, $q_1 \ge {r_1}_*\}$.
Denote the class of series $f_2$ satisfying this property by $_2\mathcal V(X)$. 

Sector $_1{\bf V}$ corresponds to the set
\begin{equation}\label{e.U}
_1\mathcal U (\e) = \left\{X: |X|^{R^*} \le \e, |x_1|\le \e\right\},
\end{equation}
and power series in $_1\mathcal V(X)$ are called convergent if they converge absolutely in some  $_1\mathcal U (\e)$.
Observe that $_1{\bf V}$ is contained in $\bf V$  and that $_1\mathcal U (\e)$ 
contains the sector $\mathcal U_{\bf V}(\e)$ given by~\eqref{e.sss}.

\medskip {\it
 
\noindent {\bf Third Normal Form} \cite[Ch. II, \S2, Thm 2, p. 134]{B}: 
If the series $f_i$ in \eqref{e:br0} are of class $_i{\mathcal V}(X)$, then there exists a
formal change of coordinates \eqref{e:br1}, where the $h_i$ are series of class $_i{\mathcal V}(Y)$, which
transforms \eqref{e:br0}  into  system \eqref{e:br2} in which the $g_i$ are series of class $_i{\mathcal V}(Y)$ 
consisting only of terms $g_{iQ}Y^Q$ satisfying $\langle Q, \Lambda \rangle =0$. 
}

\medskip

Analogous statement also holds if we interchange the role of variables $x_1$ and $x_2$.
Furthermore, it is shown in \cite{B} that the behaviour
of the integral curves of system~\eqref{e:br0} and the normal form~\eqref{e:br2} coincide in the region given by~\eqref{e.U} similarly to the Second Normal Form Theorem.

The advantage of the Third Normal Form over the Second Normal Form is that it describes the behaviour of integral
curves on a bigger region, albeit for a smaller class of power series. 

Methods of integration of systems given in the above normal forms are carefully described in~\cite{B}. This makes it possible
to construct the local phase portrait of these systems.

\subsection{The Newton diagram.} 
Let $\mathcal X$ be a real analytic vector field on $\rl^2$ given by 
\begin{equation}\label{e:sample}
\left\{
\begin{array}{c}
\dot t = \sum_{j+k >1} \alpha_{jk}t^j s^k =t f_1(t,s)\\
\dot s = \sum_{j+k >1} \beta_{jk}t^j s^k =s f_2(t,s) .
\end{array} \right.
\end{equation}
Of course, this notation for components of $\mathcal X$ is independent of the notation of Section 4 where $f$ was 
the map defined in Section~4.2. We write
\begin{equation}\label{e:f_j}
f_j(t,s) = \sum_Q f_{jQ}(t,s)^Q,
\end{equation}
where $Q = (q_1,q_2)$, and $(t,s)^Q=t^{q_1}s^{q_2}$. 
{\it The support} $\bf D$ of $\mathcal X$ is the set of points $Q = (q_1,q_2)$ in $\mathbb Z^2$ such that 
$\vert f_{1Q}\vert + \vert f_{2Q}\vert \neq 0$. Fix a vector $P \in \rl^2$ and put 
$c = \sup_{Q \in \bf D} \langle Q,P \rangle$; here $\langle \bullet,\bullet \rangle$ denotes the euclidean inner product. 
The set 
$$L_P = \{ Q \in \rl^2: \langle Q,P \rangle = c \}$$
forms the {\it support line} $L_P$ of $\bf D$ with respect   to the vector $P$, while the set 
$$L_P^{(-)} = \{ Q \in \rl^2: \langle Q,P \rangle \leq c \}$$
defines the {\it support half-space} $L_P^{(-)}$ corresponding to the vector $P$.

{\it The Newton polygon} $\Gamma$ is defined as the intersection of all support half-spaces of ${\bf D}$, i.e.,
$$
\Gamma = \bigcap_{P \in \mathbb R^2 \setminus \{0\}} L_P^{(-)}.
$$
It coincides with the closure of the convex hull of $\bf D$ (see \cite{B}). Its boundary consists of edges,
which we denote by $\Gamma_j^{(1)}$, and vertices, which we denote by $\Gamma_j^{(0)}$, where
$j$ is some enumeration. In this notation the upper index expresses the dimension of the object.

Part of the boundary of $\Gamma$, called the {\it Newton diagram} or the {\it open Newton polygon} in the terminology of \cite{B}, denoted by $\hat \Gamma$, plays an important role in the theory of power
series transformations. For simplicity we consider only the case relevant to us when $\bf D$ is contained in the set 
$\{ Q = (q_1,q_2): q_j \geq -1 , j = 1,2\}$.
Then the Newton diagram can be constructed explicitly  as follows.  
Let $q_{2*} = \min \{ q_2: (q_1,q_2) \in {\bf D}\}$. Then $x_2 = q_{2*}$ is  the horizontal support line to $\bf D$. Set  $q_{1*} = \min \{ q_1: (q_1,q_{2*}) \in {\bf D} \}$. 
The point $\Gamma^{(0)}_1:= (q_{1*},q_{2*})$ is the left boundary  point of the intersection of ${\bf D}$ with 
the horizontal support line $q_2 = q_{2*}$. Consider the support line $L_P$ for ${\bf D}$ through 
$\Gamma^{(0)}_1$ satisfying the following assumptions: 
\begin{itemize}
\item[(i)] $P = ( p_1,p_2)$ with $p_1 < 0$ and $ p_2 < 0$; 
\item[(ii)] $L_P$  contains at least one other point of ${\bf D}$. 
\end{itemize}
The first assumption means that the  line $L_p$ admits a normal vector which lies in the third quadrant. In particular, 
$L_P$ is not a horizontal or vertical line. Clearly, these two conditions define  such a support line  uniquely. If the line $L_P$ 
does not exist, our procedure stops on this first step and we set $\hat\Gamma = \{\Gamma^{(0)}_1\}$, that is the Newton 
diagram consists of a single vertex.  Otherwise denote by $\Gamma^{(0)}_2$ the left
boundary point of the intersection of ${\bf D}$ with $L_P$. Consider now the support line through
$\Gamma^{(0)}_2$ with the above properties (i) and (ii); hence, it contains a point of ${\bf D}$ different from
$\Gamma^{(0)}_1$. Continuing this procedure we arrive to the point $Q^* = (q_1^*,q_2^*)$ which is the
lowest point of ${\bf D}$ on the left vertical support line of ${\bf D}$,
i.e., $q_1^* = \min \{ q_1: (q_1,q_2) \in {\bf D } \}$ and $q_2^* = \min \{ q_2:
(q_1^*,q_2) \in {\bf D}\}$. Denote this last point by $\Gamma^{(0)}_k$. For every $j =1,...,k-1$ 
we denote by $\Gamma_j^{(1)}$ the edge joining the vertices $\Gamma_j^{(0)}$ and $\Gamma_{j+1}^{(0)}$. Thus by
construction, the points $\Gamma^{(0)}_1$ and $\Gamma^{(0)}_k$ are joined by the Newton
diagram $\hat \Gamma$.

\medskip

It is important to notice here that all edges and vertices of the Newton diagram $\hat \Gamma$
are edges and vertices of the Newton polygon $\Gamma$, but in general, not all edges and vertices of $\Gamma$ are edges and vertices of $\hat\Gamma$. Consider some examples.

\medskip

{\it Example 1.} Let $\bf D$ consist of two points $(1,1)$ and $(1,2)$. Then the Newton diagram consists of a single vertex $\Gamma^{(0)}_1 = (1,1)$.

\medskip

The next example will occur in Section 6.

\medskip

{\it Example 2.}  Let $\bf D$ consist of three points $(2,0)$, $(4,0)$ and $(0,2)$. Then the Newton diagram is formed
by two vertices $\Gamma^{(0)}_1 = (2,0)$, $\Gamma^{(0)}_2 = (0,2)$, and one edge $\Gamma^{(1)}_1$, which is 
the segment joining these vertices.

\subsection{Nonelementary singularity.}\label{s.strategy}
Bruno's method for construction of the phase portrait of a vector field near a nonelementary 
singular point can be described as follows. 
For each element $\Gamma_j^{(d)}$ of the Newton diagram associated with \eqref{e:sample}, there is a corresponding sector 
$\mathcal U^{d}_j$ in the phase space $\rl^2_{(t,s)}$, so that together they form a neighbourhood of 
the origin (here boundaries of the sectors are not necessarily integral curves). In each $\mathcal U^{0}_j$ 
one brings the system to a normal form, and in $\mathcal U^{1}_j$ one uses power transformations 
(quasihomogeneous blow-ups) to reduce the problem to the study of elementary singularities of the 
transformed system. This allows one to determine the behaviour of the orbits in each sector applying the above Normal Form theorems and using a  careful study of integral curves for all types of normal forms in~\cite{B}. After that the 
results in each sector are glued together to obtain the overall phase portrait of the system near the origin.

We now consider some important special cases corresponding to particular elements of the Newton diagram. 

{\bf Case of a vertex.}  Let $Q =\Gamma^{(0)}_j$ be  a vertex of the Newton diagram. Consider the edges 
$\Gamma_{j-1}^{(1)}$ and   $\Gamma_{j}^{(1)}$ adjacent to $Q$ in the Newton diagram. Next, consider the 
{\it unit}  (i.e., their coordinates are coprime integers)  vectors $R_{j-1}=(r_{1,j-1},r_{2,j-1})$ and 
$R_{j}=(r_{1,j},r_{2,j})$ directional to $\Gamma_{j-1}^{(1)}$ and   $\Gamma_{j}^{(1)}$ respectively. 
We impose here the restrictions $r_{2,j-1} > 0$ and $r_{2,j} > 0$ so these vectors are determined uniquely. 
Set $R_* = -R_{j-1}$ and $R^* = R_j$. In the special case when $Q$ is a boundary point of $\hat \Gamma$, one 
of the adjacent edges does not exist, so if $Q$ is the right boundary point  $Q_*$, we set $R_* = (1,0)$, and if $Q$ is the left 
boundary point  $Q^*$, we put $R^* = (0,1)$.

The method of \cite{B} associates to $Q$ a set defined by 
\begin{equation}
\mathcal U^{(0)}_j (\e) = \{ (t,s)\in \rl^2 : (|t\vert,\vert s|)^{R^*} \le \e, \ (|t\vert,\vert s|)^{R_*} \le \e, |t\vert \le \e, |s| \le \e  \},
\end{equation}
for some $\e > 0$.  System~\eqref{e:sample}, after the change of the old time variable $\tau$ with the new 
time variable $\tau_1$ satisfying $d\tau_1=(t,s)^{Q}d\tau$, is of form \eqref{e:br0}. Furthermore, the vectors $R^*$ and $R_*$ defined above by the adjacent edges at $Q$,  will generate for this  new system  \eqref{e:br0} the convex cone ${\bf V}$ as described  in the previous subsection, so the notation is consistent. The obtained system satisfies the assumptions of the Principal or the Second Normal Form Theorem. The behaviour of the integral curves of the normal form and the original system coincides in $\mathcal U^{(0)}_j (\e)$ for $\e$ sufficiently small.

A particularly simple case occurs when $Q=(q_1,q_2)=\Gamma^{(0)}_j$ is the first (i.e., the right) or the last (i.e., the left) 
point of $\hat \Gamma$, and $Q$ is not contained in the first quadrant (Type I according to classification in~\cite[p.~138]{B}). 
In this situation one of the coordinates of $Q$ equals $-1$. Say, if $q_2 = -1$,  i.e., $ Q$ is the right point of $\hat\Gamma$, then one takes $R_*=(1,0)$ according to the general rule stated above. 
The corresponding normal form has vertical integral curves. It follows that the original 
system~\eqref{e:sample} in the set
$$
\mathcal U_*(\e)=\{(t,s)\in \rl^2 : (|t\vert,\vert s|)^{R^*} \le \e, |t|\le \e \}
$$
does not have any integral curves terminating at the origin. Similarly, if $q_1=-1$, i.e., if $Q$ is the left point of $\Gamma$, then $R^*=(0,1)$, 
and again in 
$$
\mathcal U^*(\e)=\{(t,s)\in \rl^2 : (|t\vert,\vert s|)^{R_*} \le \e, |s|\le \e \}
$$
the system does not have any characteristic orbits.

{\bf Case of an edge. } Suppose now that $\Gamma^{(1)}_j$ is an edge of $\hat \Gamma$. Let $R=(r_1, r_2)$ , $r_2 > 0$ be a unit directional
vector of $\Gamma^{(1)}_j$. The corresponding set in the phase space is given by
\begin{equation}
\mathcal U^{1}_j(\e) = \{(t,s)\in \rl^2 : \e \le (|t\vert,\vert s|)^R \le 1/\e, \ \ |t|\le \e, |s| \le \e \}.
\end{equation}
Consider the power transformation given by $y_1 = t^{k_1} s^{k_2}, y_2 = t^{r_1}s^{r_2}$, where the integers 
$k_1,k_2$ are chosen such that the matrix 
\begin{equation}
A = \left(
\begin{array}{cc}
k_1 & k_2\\
r_1 & r_2
\end{array}\right)
\end{equation}
has the determinant equal to 1. In the matrix form, we can write $X=(t,s)$, 
$$Q = \left(\begin{array}{c} q_1 \\ q_2  \end{array}\right),$$ 
$$F_Q = \left(\begin{array}{c} f_{1q} \\ f_{2q} \end{array}\right).$$
Then~\eqref{e:sample} can be given by 
\begin{equation}\label{e:ex1}
\dot{(\ln X)} = \sum_{Q\in {\bf D}} F_Q X^Q,
\end{equation}
where $X^Q = t^{q_1}s^{q_2}$. The power transformation can be expressed now as $Y=X^A$ 
taking~\eqref{e:ex1} into
$$
\dot{(\ln Y)} = \sum_{Q'\in D'} F'_{Q'} Y^{Q'},
$$
with $Y=(y_1,y_2)$, $Q'=(A^t)^{-1}Q$, $D'=(A^t)^{-1}{\bf D}$ (the superscript $t$ stands for 
transposition), and $F'_{Q'}=AF_Q$. After division by the maximal power of $y_1$ one obtains a
new system. Here the $y_2$-axis corresponds to $\{t=s=0\}$ in the original coordinates, and therefore 
one needs to investigate the new system in a neighbourhood of the $y_2$-axis.  Quite often the topological 
behaviour of the system in $\mathcal U^{1}_j(\e)$ can be determined 
by considering the {\it truncation} of the system which is obtained by taking the sum in~\eqref{e:f_j} only 
over the vertices contained in $\Gamma^{(1)}_j$. The detailed discussion is in \cite{B}, pp. 140-141.  For instance, 
in the situation which we will encounter below, the truncated system will have an elementary singularity. In 
general, the singularities of the new system can be nonelementary, but they are simpler than those of the original 
system. Therefore, the general method described above can be applied  and  an induction procedure can
be used.

We do not go into further details since the goal of this section is just to outline the strategy of the employed method. 
The computations of the next sections will strictly follow the presented method and, as we hope, will clarify the details.

\section{Phase portrait of the standard umbrella}\label{s:gen-hyp-s}

Since the standard umbrella corresponds to the nongeneric case where $\phi$ is the identity map, we study its 
characteristic foliation separately. We rewrite system (\ref{e:hyp-s}) in the form 
\begin{equation}\label{e:hyp1}
\left\{
\begin{array}{l}
\dot t = t(-3t^2-s^2-3t^4) = tf_1(t,s)\\
\dot s = s(s^2+4t^2  + 7 t^4) = sf_2(t,s), 
\end{array}\right.
\end{equation}
and set
$$f_j(t,s) = \sum_Q f_{jQ}(t,s)^Q,$$
where $Q = (q_1,q_2)$ is  the multi-index with integer entries, and $(t,s)^Q=t^{q_1}s^{q_2}$.

\begin{figure}[htb]
\begin{center}
\leavevmode
\includegraphics[scale=0.6, trim = 20mm 30mm 20mm 10mm, clip]{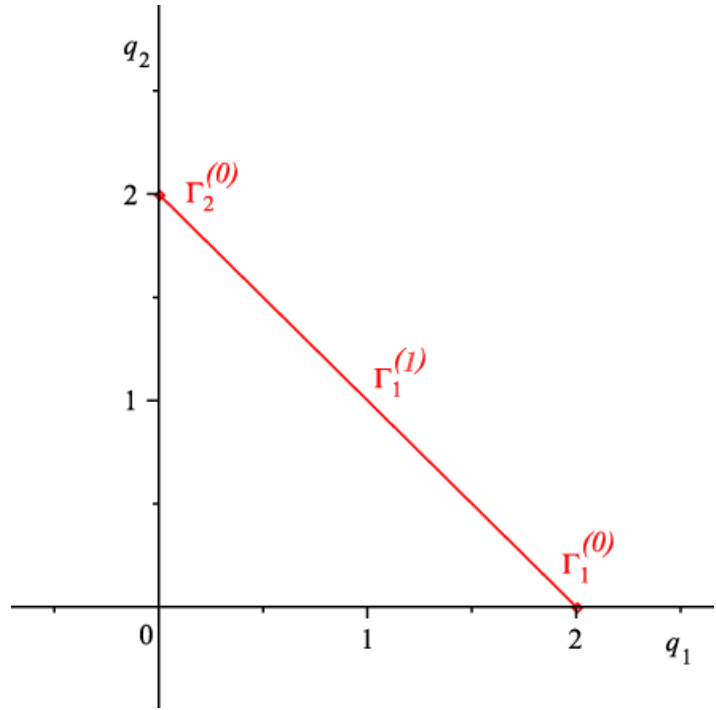}
\end{center}
\caption{The Newton diagram for \eqref{e:hyp1}.}
\label{fig1}
\end{figure}

The Newton diagram $\hat\Gamma$ consists of two vertices $\Gamma^{(0)}_1 = (2,0)$ and
$\Gamma^{(0)}_2 = (0,2)$ and the line segment (edge) $\Gamma_1^{(1)}$  between them (see Fig.~\ref{fig1}). We point out that the point $(4,0)$ lies in the support  $\bf D$  but does not belong to the Newton diagram $\hat \Gamma$.
For each element of the Newton diagram (the two vertices and the edge), there is a corresponding sector in the 
phase space $\rl^2_{(t,s)}$, so that together they form a neighbourhood of the origin. Accordingly we consider 
3 cases.

\medskip

{\it Case 1.} First consider the vertex $(2,0)$. Following the strategy outlined in Section~\ref{s.strategy}, 
we set $R_*=(1,0)$, and $R^*=(-1,1)$. We can make the 
change of time $d \tau_1 =  t^2 d \tau$. This yields the system
\begin{equation}\label{e:h1}
\left\{\begin{array}{l}
\frac{dt}{d\tau_1} = -t (3+t^{-2}s^{2}+3t^2)= -3t + t f_1(t,s)\\
\frac{ds}{d\tau_1} = s (4 + t^{-2}s^2 + 7t^2) = 4s + s f_2(t,s).
\end{array}\right.
\end{equation}
The Newton diagram ${\hat \Gamma}$  corresponding to  \eqref{e:h1} has vertices
$(-2,2)$ and $(2,0)$, in particular, it is 
contained in the sector $\bf V$ (with the angle $< \pi$) bounded by the rays generated by 
$R_*$ and $R^*$. Therefore, for sufficiently small $\e$, in the sector
$$
{\mathcal U}^{(0)}_1 = \{ (t,s)\in \rl^2: (\vert t \vert ,\vert s \vert)^{R_*} \le \e,\ (\vert
t \vert,\vert s \vert)^{R^*} \le \e \} 
= \{ |t| \le \e,\  |s|\le \e |t|\},
$$
there exists a smooth change of variables $(t,s)$ putting the initial system to  the Second Normal 
Form of Bruno. In the new coordinates  the system has the form
\begin{equation}\label{e:ttt}
\left\{\begin{array}{l}
\dot{y_1} = -3y_1+ y_1 \sum g_{1Q} (y_1,y_2)^Q\\
\dot{y_2} = 4 y_2 + y_2 \sum g_{2Q} (y_1,y_2)^Q,
\end{array}\right.
\end{equation}
where the coefficients $g_{1Q}$ and $g_{2Q}$ are 
all zero except those for which $-3 q_1 + 4q_2 = 0$. The line $L:= \{ -3y_1
+ 4y_2=0\}$ determined by the linear part of system~\eqref{e:ttt} intersects the interior of the 
sector~$\bf V$ (see Fig.~\ref{fig2}).  It follows (see Bruno \cite{B}, p. 132) that the system defined 
by \eqref{e:ttt}, and hence by \eqref{e:h1}, is a saddle, i.e., each ray $\{y_1 = 0, y_2>0\}$, $\{y_1 > 0, y_2=0\}$ is an 
integral curve, and in each quadrant in $\rl^2$, the integral lines are homeomorphic 
to hyperbolas. This is the description of system~\eqref{e:hyp}
in sector ${\mathcal U}^{(0)}_1$.

\begin{figure}[htb]
\begin{center}
\leavevmode
\includegraphics[scale=0.6, trim = 5mm 20mm 20mm 20mm, clip]{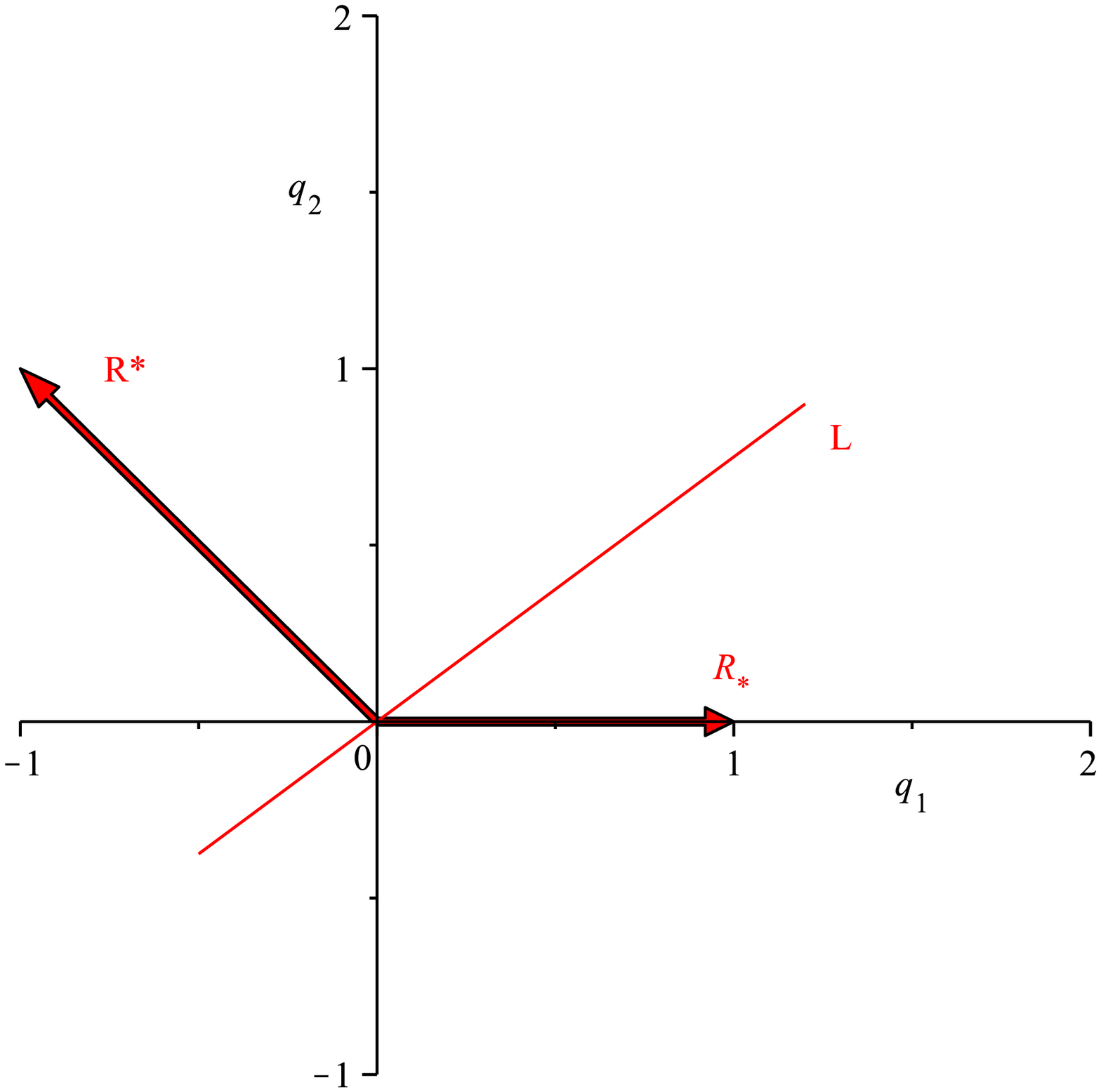}
\end{center}
\caption{Case 1 for \eqref{e:hyp1}.}
\label{fig2}
\end{figure}

\medskip

{\it Case 2.} Consider now the second vertex $(0,2)$.  Here we have 
$R_*=(1,-1)$  and $R^*=(0,1)$.  The corresponding sector where the 
change of dependent variables will be performed is given by
$$
{\mathcal U}^{(0)}_2 =  \{ (t,s)\in \rl^2: (\vert t \vert ,\vert s \vert)^{R_*} \le \e,\ (\vert
t \vert,\vert s \vert)^{R^*} \le \e \}  = 
\left\{|s| \le \e,\  |s|\ge \frac{|t|}{\e} \right\}.
$$
The change of time $d\tau_1 = s^2 d\tau$ transforms
system~\eqref{e:hyp} into
\begin{equation}\label{e:h2}
\left\{\begin{array}{l}
\frac{dt}{d\tau_1} = -t +  t(3t^{2}s^{-2}+3t^4s^{-2})\\
\frac{ds}{d\tau_1} = s  + s(4t^{2}s^{-2} + 7t^4s^{-2}).
\end{array}\right.
\end{equation}
As above, there exists a smooth change of variables $(t,s)$ putting this system
to  the second normal form:
$$
\left\{\begin{array}{l}
\dot{y_1} = -y_1+ y_1 \sum g_{1Q} (y_1,y_2)^Q\\
\dot{y_2} = y_2 + y_2 \sum g_{2Q} (y_1,y_2)^Q,
\end{array}\right.
$$
where the coefficients $g_{1Q}$ and $g_{2Q}$ are 
all zero except those which belong to the line  $L:= \{ -q_1 +q_2 = 0\}$. This line intersects the
sector $\bf V$ bounded by $R_*$ and $R^*$ which implies  that this system is 
again a saddle. This gives the phase portrait of~\eqref{e:hyp} in sector ${\mathcal U}^{(0)}_2$

\begin{figure}[htb]
\begin{center}
\leavevmode
\includegraphics[scale=0.6, trim = 5mm 5mm 20mm 30mm, clip]{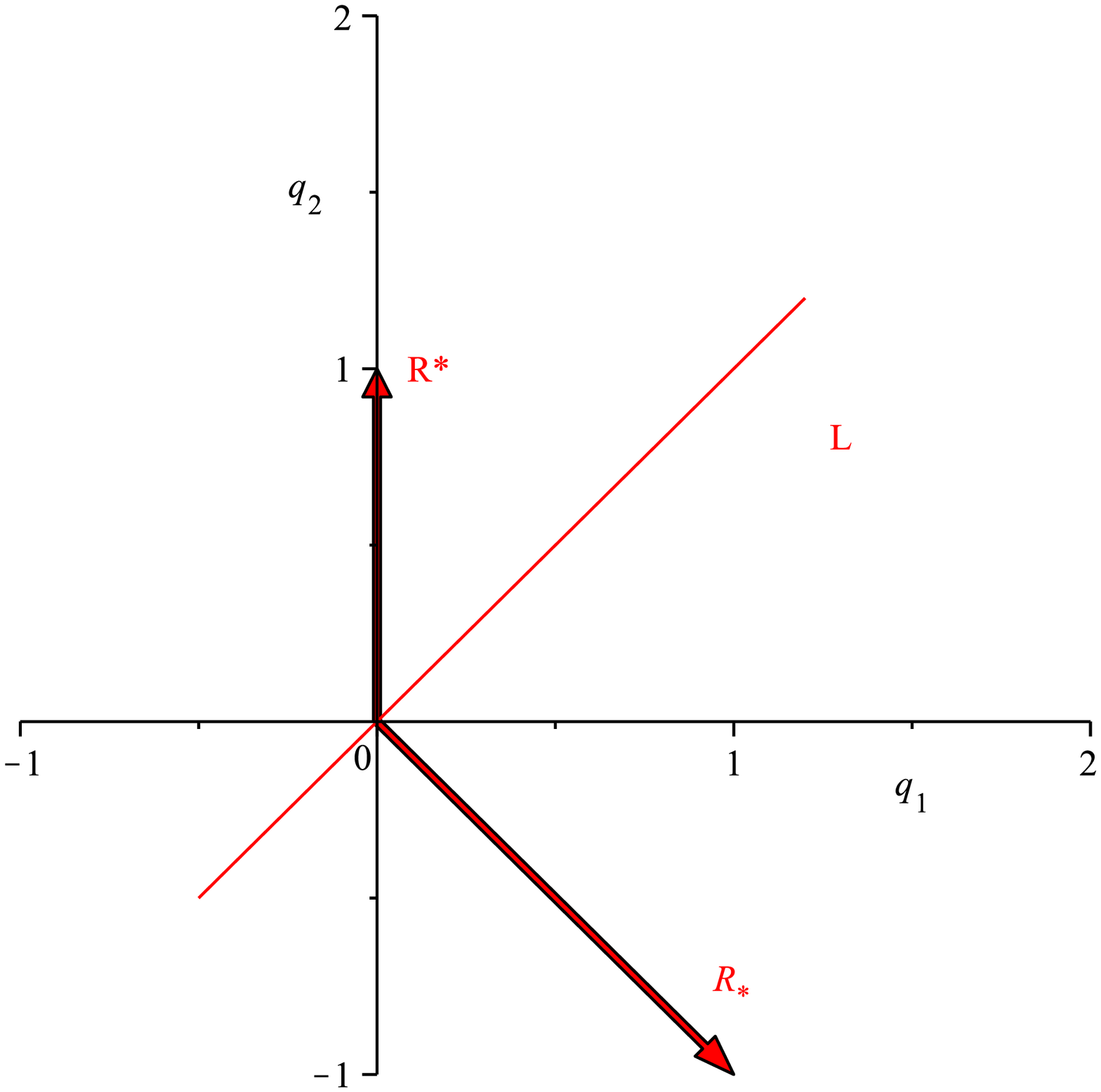}
\end{center}
\caption{Case 2 for \eqref{e:hyp1}.}
\label{fig3}
\end{figure}

\medskip

{\it Case 3.} The remaining case of the edge between $(2,0)$ and $(0,2)$ will
correspond to the sector ${\mathcal U}^{(1)}_1$, which is the complement of 
${\mathcal U}^{(0)}_1 \cup {\mathcal U}^{(0)}_2$. We make the following change of variables
\begin{equation}\label{e:ch3-s}
\left\{\begin{array}{l}
y_1 = t,\\
y_2 = t^{-1}s.
\end{array}\right.
\end{equation}
In the matrix form, we write $X=(t,s)$, and the change
of variables~\eqref{e:ch3-s} can be expressed as $Y=X^A$ with the 
matrix of exponents
$$
A = \left(\begin{array}{cc}
1 & 0 \\ -1 & 1 
\end{array}\right).
$$
Then system~\eqref{e:hyp} takes the form
\begin{equation*}
\left\{\begin{array}{l}
\dot{y_1} = y_1(-3y_1^2 - y_1^2y_2^2-3y_1^4) \\
\dot{y_2} = y_2(7y_1^2+2y_1^2y_2^2+10y_1^4).
\end{array}\right.
\end{equation*}
The edge of $\hat\Gamma$ becomes vertical in the new system.
Performing as above a change of time, we may divide both sides by $y_1^2$ to obtain
\begin{equation}\label{e:h3}
\left\{\begin{array}{l}
\dot{y_1} = -3y_1-y_1y_2^2-3y_1^3 = y_1(-3- y_2^2-3y_1^2) \\
\dot{y_2} = 7y_2+2y_2^3+10y_1^2y_2 = y_2(7+2y_2^2+10y_1^2).
\end{array}\right.
\end{equation}
Under the change of variables~\eqref{e:ch3-s}, the line $y_1=0$
corresponds to the origin, and therefore, we are interested in the 
integral curves of system~\eqref{e:h3} that intersect the line $y_1=0$
at points with $y_2\ne 0$. The set $\{y_1=0,\ \pm y_2>0\}$ are integral curves
of~\eqref{e:h3}, but they correspond to $t=s=0$ in the original system.
According to Bruno (\cite{B}, p. 141), the
points on the $y_2$ axis can be either simple points, in which case
the integral curves of~\eqref{e:h3} near such points are parallel to
the $y_2$-axis, or singular points. The truncation of system \eqref{e:h3}  (see the end of the 
previous section) contains only the terms that correspond to the edge under consideration and its 
vertices, and thus has the form 
\begin{equation}\label{e:h4}
\left\{\begin{array}{l}
\dot{y_1} =  y_1\hat f'_{\hat{10}}(y_2) \\
\dot{y_2} = y_2\hat f'_{\hat{20}}(y_2),
\end{array}\right.
\end{equation}
where $\hat f'_{\hat{20}}(y_2)= 7+2y_2^2$ (we follow the notation of \cite{B}). 
Singular points are determined
from the equation $\hat f'_{\hat{20}}(y_2)=0$.  In 
our case  $\hat f'_{\hat{20}}(y_2)$  is strictly positive.
Therefore, in~\eqref{e:h4} all points with $y_1=0, y_2\ne 0$ are
simple points. From this we conclude that in the sector ${\mathcal U}^{(1)}_1$ no
integral curves of system~\eqref{e:hyp} intersect the origin.

\begin{figure}[htb]
\begin{center}
\leavevmode
\includegraphics[scale=0.6, trim = 6mm 6mm 6mm 6mm, clip]{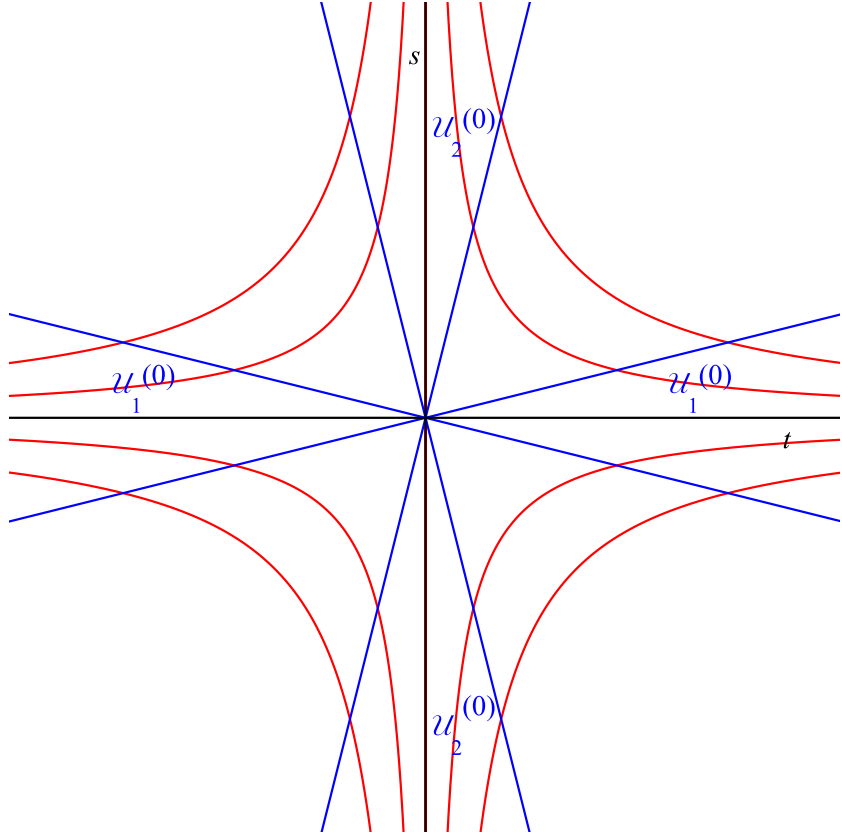}
\end{center}
\caption{Phase portrait of \eqref{e:hyp1}.}
\label{fig4}
\end{figure}

With this information the integral curves in all sectors can be glued together. It is readily verified that
the phase portrait of system~\eqref{e:hyp} is in fact a saddle,  the integral curves in each quadrant 
of $\rl^2$ are homeomorphic to hyperbolas and do not intersect the coordinate axes (see Fig.~\ref{fig4}).

\section{Phase portrait of umbrella in general position}\label{s:gen-hyp}

We now perform similar calculations for the algorithm to determine the topological structure near the origin 
of the dynamical system  defined by (\ref{e:system}). First of all we represent it in the canonical form

\begin{equation}\label{e:sup}
\left\{\begin{array}{l}
\dot t = t(-2g_{12}s + \alpha_{02}t^{-1}s^2 - 3g_{22}t^2 + o(\vert s \vert + \vert t^{-1}s^2 \vert + \vert t \vert^2 ))\\
\dot s = s(4g_{11}t^2 + \beta_{12}ts + \beta_{03}s^2+ 6g_{12}t^4s^{-1} + o(\vert t^2 \vert + \vert ts \vert + \vert s^2 \vert + \vert t^4 s^{-1} \vert)).
\end{array}\right.
\end{equation}
The Newton diagram $\hat\Gamma$ consists of 3 vertices $(-1,2)$,  $(0,1)$ and $(4,-1)$, and the two 
edges between them (Fig.~\ref{fig5}). Five cases should be considered each corresponding to a vertex 
or an edge of $\hat \Gamma$.

\begin{figure}[htb]
\begin{center}
\leavevmode
\includegraphics[scale=0.6, trim = 4mm 10mm 2mm 4mm, clip]{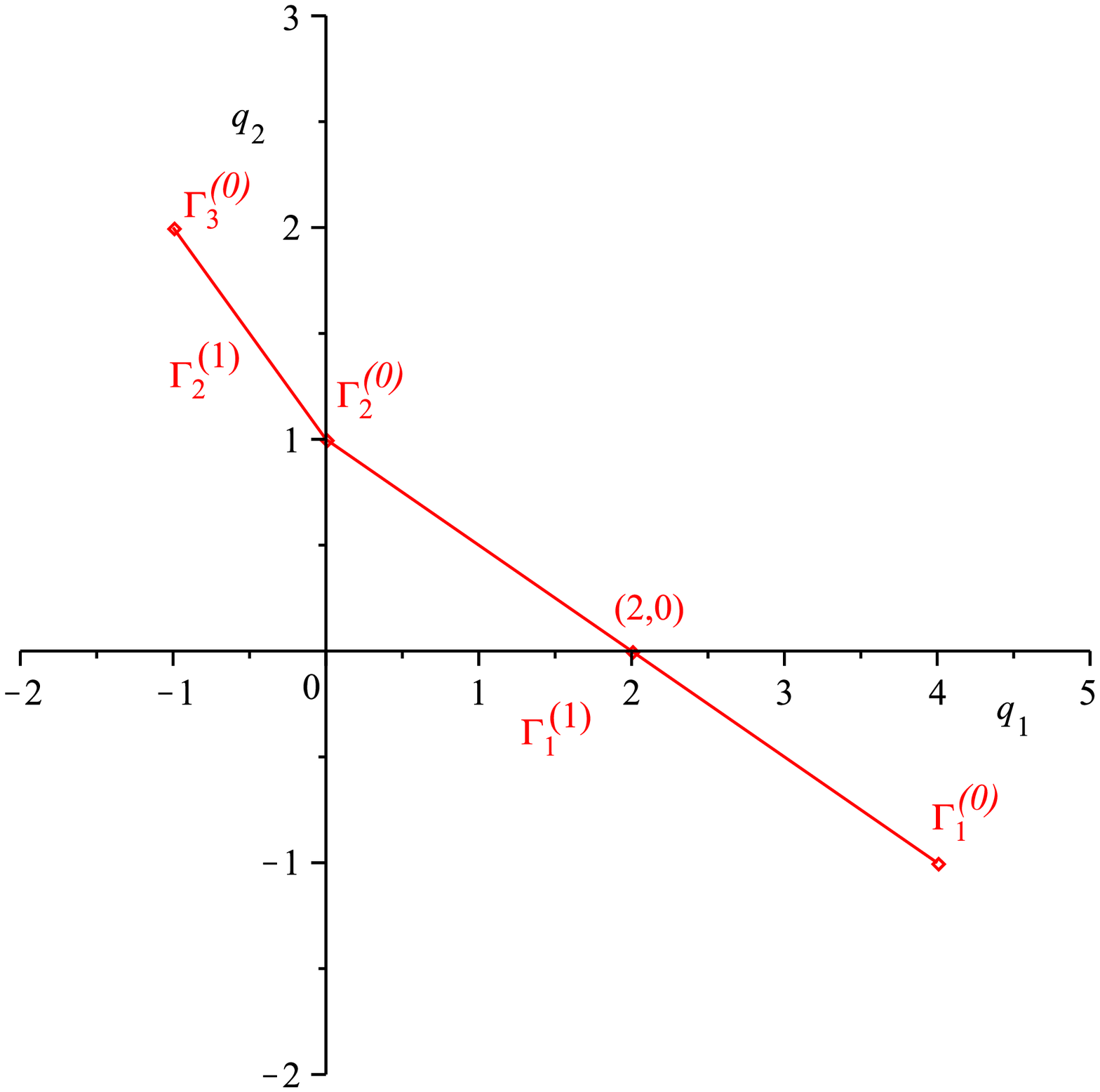}
\end{center}
\caption{The Newton diagram for \eqref{e:sup}.}
\label{fig5}
\end{figure}

\medskip

{\it Case 1.} Vertex $(4,-1)$. This corresponds to the situation discussed in Section~\ref{s.strategy}. We obtain immediately the behaviour of integral curves of the system. Namely, in the sector
$$
\mathcal U_1^{(0)} = \{ (t,s)\in \rl^2: (\vert t \vert ,\vert s \vert)^{(1,0)} \le \e,\ (\vert
t \vert,\vert s \vert)^{(-2,1)} \le \e \} 
= \{ |t| \le \e,\  |s|\le \e|t|^2\}
$$
the integral curves are vertical, in particular, there are no curves terminating at the origin.

\medskip

{\it Case 2.} Vertex $(-1,2)$. Again the same analysis works here. Since $(-1,2)$ is the end
point of $\hat \Gamma$, i.e., of Type I in \cite[p.~138]{B}, it follows from \cite{B} that in 
$$
\mathcal U_3^{(0)} = \{ (t,s)\in \rl^2: (\vert t \vert ,\vert s \vert)^{(0,1)} \le \e,\ (\vert
t \vert,\vert s \vert)^{(1,-2)} \le \e \} 
= \{ |s| \le \e,\  |t|\le \e|s|^2\}
$$
the integral curves are horizontal, and no curves terminate at the origin.

\medskip

{\it Case 3.} Vertex $(0,1)$. This is Type III in \cite[p.~139]{B}.
After a change of time
so that $d\tau_1 = s d\tau$, the system takes the form
\begin{equation}\label{e:c3}
\left\{\begin{array}{l}
\dot t = t(-2g_{12} + \alpha_{02}t^{-1}s - 3g_{22}t^2 s^{-1}+ o(1 + |t^{-1}s|+ \vert t^2s^{-1}|))\\
\dot s = s(4g_{11}t^2 s^{-1} + \beta_{12}t + \beta_{03}s+ 6g_{12}t^4s^{-2} +o(|t^2s^{-1}|+|t|+|s|
+|t^4s^{-2}|)).
\end{array}\right.
\end{equation}
There are two sectors which can be assigned to vertex $(0,1)$. One of them is determined by
$R_* = (2,-1)$ and $R^*=(-1,1)$, and equals 
$$
\mathcal U_2^{(0)} = \{ (t,s)\in \rl^2: (\vert t \vert ,\vert s \vert)^{R^*} \le \e,\ (\vert
t \vert,\vert s \vert)^{R_*} \le \e \}.
$$
We may apply here the Second Normal Form of Bruno. Since we consider a generic case,  we have  
$\lambda_1 = - 2g_{12} \neq 0$. Further, $\lambda_2=0$, because the second equation has no free term. 
Recall that we use the notation $\Lambda = (\lambda_1, \lambda_2)$. The line $L$  determined by
\begin{equation}\label{e:lam}
L = \{ Q = (q_1,q_2) \in \rl^2: \langle Q, \Lambda \rangle=0 \} = \{ q_1 = 0 \}
\end{equation}
enters the interior of the sector bounded by $R^*$ and $R_*$. It follows that in $\mathcal U_2^{(0)}$ 
there are no integral curves terminating at the origin.

On the other hand, we may use the Third Normal Form of Bruno for \eqref{e:c3}. It
is valid on a bigger domain, namely, on 
$$
{}_2 \mathcal U_2^{(0)} = \{ (t,s)\in \rl^2: (\vert t \vert ,\vert s \vert)^{R_*} \le \e, |s|\le \e\} =
\{|t|^2 \le \e |s|, |s| \le \e\}.
$$
The region of the $(t,s)$-space where the dynamics takes place is given by 
$$
{}_2{\bf V} = \{Q: Q= a_1 R_* + a_2 \cdot (0,1),\ \  a_1, a_2 \ge 0\}.
$$
Now the line $L$ determined from \eqref{e:lam} enters ${}_2{\bf V}$ along its boundary, the $s$-axis. 
In general, this yields a complicated behaviour of the system in ${}_2 \mathcal U_2^{(0)}$. In fact, there
are four possibilities as described in \cite[p. 134 Case c)]{B}. So which case is it? The salvation comes from 
{\it Case 2} above: it describes the behaviour of the system in $\mathcal U_3^{(0)}$
(which is a subset of ${}_2 \mathcal U_2^{(0)}$ and a neighbourhood of the $s$-axis). According to {\it Case 2}, 
the integral curves are horizontal near the $s$-axis, which eliminates all possibilities but one. We conclude 
that no integral curves enter the origin in ${}_2\mathcal U_2^{(0)}$.

\medskip

{\it Case 4.} Edge connecting $(0,1)$ and $(-1,2)$. The corresponding sector is defined by 
$$\mathcal U_2^{(1)} = \{ Q \in \rl^2: \e \le (\vert t \vert, \vert s \vert)^{(-1,1)} \le \e^{-1} \}$$
(see \cite[p. 139]{B}). This case is subsumed by Case~3 above because 
$\mathcal U_2^{(1)} \subset {}_2\mathcal U_2^{(0)}$ in a suitable neighbourhood of the origin.

\begin{figure}[htb]
\begin{center}
\leavevmode
\includegraphics[scale=0.6, trim = 2mm 4mm 2mm 4mm, clip]{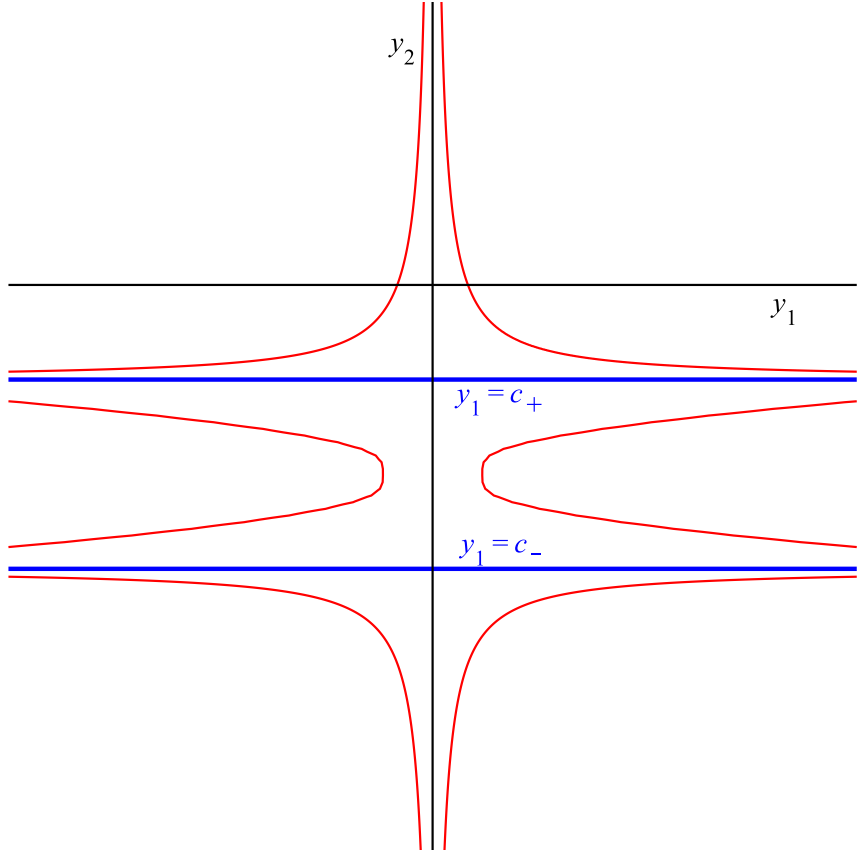}
\end{center}
\caption{Phase portrait in $y$-coordinates for $g_{12} > 0$.}
\label{fig6}
\end{figure}

\medskip

{\it Case 5.} Edge connecting $(0,1)$ and $(4,-1)$. We will consider the truncation of system \eqref{e:system},
i.e., we keep only terms that are related to the edge under consideration. We have
\begin{equation}\label{e:**}
\left\{
\begin{array}{l}
\dot t = t(-2g_{12}s-3g_{22}t^2)\\
\dot s = s(-4g_{11}t^2+6g_{12}t^4s^{-1}).
\end{array}\right.
\end{equation}
The directional vector is $R=(-2,1)$, and the sector in which the dynamics should be understood is 
\begin{equation}\label{e:U}
\mathcal U_1^{(1)}=\{(t,s) : \e \le (\vert t \vert,\vert s \vert)^{(-2,1)} \le \frac{1}{\e}, |t|, |s| \le \e\} = 
\{ \e |t|^2 \le |s|, \ |s| \le \frac{1}{\e}|t|^2, \ |t|\le \e, \ |s|\le \e\}.
\end{equation}
We need to make the following change of coordinates:
\begin{equation}\label{e:*}
\left\{
\begin{array}{l}
y_1=t\\
y_2=t^{-2}s,
\end{array}\right.
\end{equation}
which corresponds to the matrix 
$$
A = \left(\begin{array}{cc}1 & 0\\ -2 & 1\end{array}\right).
$$
In the new coordinates system \eqref{e:**} becomes
\begin{equation}
\left\{
\begin{array}{l}
\dot y_1 = y_1(-2g_{12}y_1^2y_2-3g_{22}y_1^2)\\
\dot y_2 = y_2(4g_{12}y_1^2y_2+(6g_{22}+4g_{11})y_1^2+6g_{12}y_1^2y_2^{-1}).
\end{array}\right.
\end{equation}
We divide by the maximal power of $y_1$, which equals 2 in this case, by performing the change of the independent 
variable: $ d\tau_1 = y_1^2 d\tau$. This yields
\begin{equation}
\left\{
\begin{array}{l}\label{e:***}
\dot y_1 = y_1(-2g_{12}y_2-3g_{22})\\
\dot y_2 = y_2(4g_{12}y_2+(6g_{22}+4g_{11})+6g_{12}y_2^{-1}).
\end{array}\right.
\end{equation}
This is the system of  Type I in \cite[p.~125]{B}. The $y_2$-axis is an integral curve, but it corresponds
to the origin in  \eqref{e:**}. Consider first the points  where the expression 
$4g_{12}y_2^2+(6g_{22}+4g_{11})y_2+6g_{12}$
is not zero; the integral curves near such a point are parallel to the $y_2$-axis. Going back to the original system via the 
inverse transformation to~\eqref{e:*}, we see that the $y_2$-axis blows down to the origin. Hence, these  integral curves  
do not terminate at zero in the original system. Now we need to investigate the situation near points where the above 
expression vanishes. For this we solve the quadratic equation
\begin{equation}\label{e:q}
2g_{12}y_2^2+(3g_{22}+2g_{11})y_2+3g_{12} =0.
\end{equation}
The discriminant of this equation  is 
$$
\mathcal D = 4g_{11}^2+9g_{22}^2+12g_{11}g_{22}-24g_{12}^2.
$$
Since $4g_{11}^2+9g_{22}^2 \ge 12 g_{11}g_{22}$, it follows that 
$\mathcal D \ge 24g_{11}g_{22} - 24g_{12}^2 = 24 \Delta >0$.  Here $\Delta$ is defined by \eqref{e:det}. Thus, equation \eqref{e:q} always has 
two simple roots:
$$
c_\pm = \frac{-(3g_{22}+2g_{11}) \pm \sqrt{4g_{11}^2+9g_{22}^2+12g_{11}g_{22}-24g_{12}^2}}{4g_{12}} .
$$
(since we consider the generic case, we can assume that $g_{12} \neq 0$). We point out that $c_\pm$ are either
both positive or both negative.  

We  need to investigate the dynamics near each point $(0,c_\pm)$. For that we first need to translate $c_\pm$
to the origin via
$$
z_1=y_1,\  y_2 = c_\pm + z_2. 
$$
In the new coordinates the system becomes
\begin{equation}
\left\{
\begin{array}{l}\label{e:5*}
\dot z_1 = z_1(-(2g_{12}c_\pm +3g_{22})-2g_{12}z_2)\\
\dot z_2 = z_2 ((8g_{12}c_\pm+6g_{22}+4g_{11})+4g_{12}z_2).
\end{array}\right.
\end{equation}
This is a system for which the origin in an elementary singularity (the linear part is not zero). 
To determine the dynamics we need to understand the sign of the coefficients of the linear
part, i.e., of
$$
\lambda_1 = -(2g_{12}c_\pm +3g_{22}) = -\frac{3}{2} g_{22}+g_{11} \mp \frac{1}{2}
\sqrt{{4g_{11}^2+9g_{22}^2+12g_{11}g_{22}-24g_{12}^2}}
$$
and 
$$
\lambda_2 = 8g_{12}c_\pm+6g_{22}+4g_{11} = \pm 2 \sqrt{{4g_{11}^2+9g_{22}^2+12g_{11}g_{22}-24g_{12}^2}}.
$$

 \begin{figure}[htb]
\begin{center}
\leavevmode
\includegraphics[scale=0.6, trim = 4mm 4mm 2mm 4mm, clip]{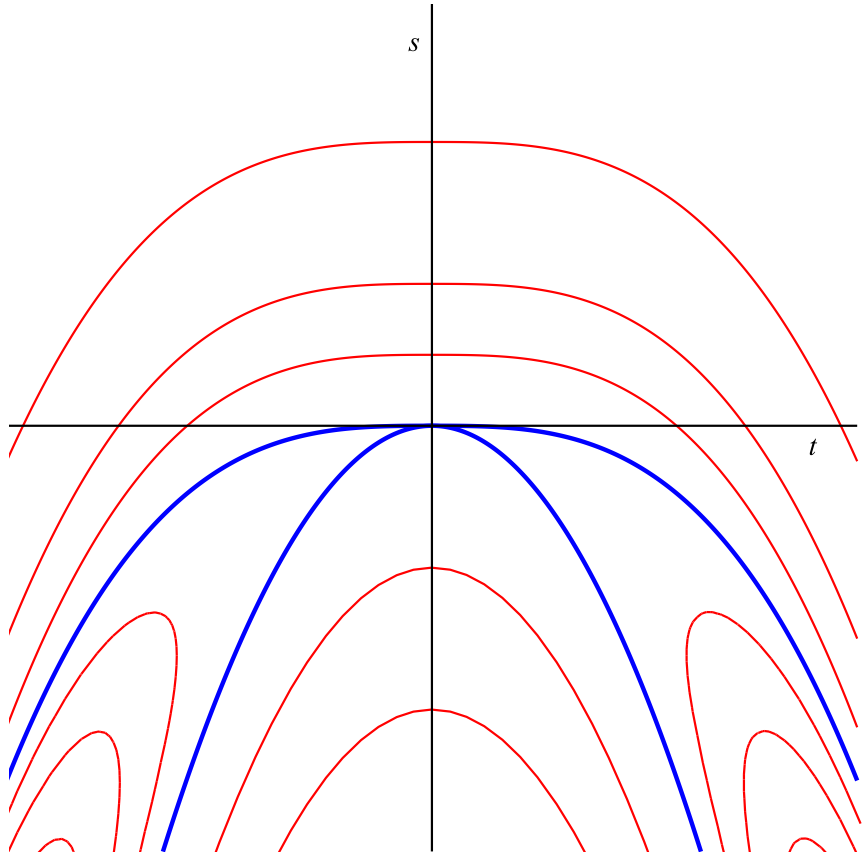}
\end{center}
\caption{Phase portrait of \eqref{e:sup}, $g_{12}>0$.}
\label{fig7}
\end{figure}
 
{\bf Claim.}  {\it $\lambda_1$ and $\lambda_2$ are of the opposite sign both for $c_+$ and $c_-$.}

\medskip 
 
First note that $\lambda_1$ and $\lambda_2$ depend only on the coefficients $g_{jk}$, i.e., only on the
linear part of the map $\psi \circ \phi$. Therefore, it is enough to prove the claim for linear symplectomorphisms.
If $\phi$ is the identity map, then it is easy to see that $\lambda_1$ and $\lambda_2$ are of the
opposite sign. 

Suppose that for some linear symplectic map $\phi_0$, the sign of $\lambda_1$ and $\lambda_2$ is the same. 
Since the symplectic group is connected, there is a path $\gamma \subset {\rm Sp}(4,\rl)$ connecting the identity 
and $\phi_0$, and  since $\lambda_j$ depend continuously on $\phi$, there exists a symplectic map on 
$\gamma$ for which one of the $\lambda_j$ is zero. Since $\mathcal D>0$, it has to be $\lambda_1$.
So $ -\frac{3}{2}g_{22} + g_{11} = \pm \frac{1}{2}\sqrt{\mathcal D}$. Therefore,
$$
4g_{11}^2 - 12g_{11}g_{22} +9g_{22}^2 = {4g_{11}^2+9g_{22}^2+12g_{11}g_{22}-24g_{12}^2}.
$$
This implies that $\Delta=0$ -- contradiction. This proves the claim. 

\medskip

Since $\lambda_j$ are of different sign, it follows that both for $c_+$ and $c_-$,  system~\eqref{e:5*} 
is a saddle at the origin. Now we are able to describe the overall dynamics in $\mathcal U_1^{(1)}$. In 
$(y_1,y_2)$- coordinates we have the following: $y_2$-axis as well as the lines $y_2=c_+$ and $y_2=c_-$ 
are the integral curves.  More precisely, the integral curves are six half-lines: $L_1 = \{ (y_1, c_+), y_1 > 0 \}$, 
$L_2 = \{ (y_1, c_+), y_1 < 0 \}$ , $L_3 = \{ (y_1, c_-), y_1 > 0 \}$, 
$L_4 = \{ (y_1, c_-), y_1 < 0 \}$, $L_5 = \{ (0, y_2): y_2 > c_+\}$, $L_6 = \{ (0, y_2): y_2 < c_-\}$, and 
one interval $I = \{ (0,y_2):   \min\{c_-, c_+\} < y_2 < \max\{c_-,c_+\} \}$ . The phase portraits near the points 
$(0, c_+)$ 
and $(0, c_-)$ are saddles, whose orbits in between the lines $y_2=c_+$ and $y_2=c_-$ 
are glued together, and are asymptotic to $L_1$ , $L_3$ or to  $L_2$, $L_4$ ; they do not touch $I$ .
Other orbits are asymptotic to $L_2$, $L_5$ or to $L_5$, $L_1$ or to $L_6$, $L_4$ or, finally, to $L_6$, $L_3$
(see Fig.~\ref{fig6}).
Going back to the original system via the inverse transformation 
to~\eqref{e:*}, we see that the $y_2$-axis blows down to a point, and we have two integral curves
$s=c_\pm t^2$ entering the origin, while other integral curves are contained in the compliment of these
two curves.  Now, if we choose $\e>0$ sufficiently small in \eqref{e:U}, we see that both curves 
$s=c_\pm t^2$ enter $\mathcal U_1^{(1)}$. This completes {\it Case 5.}

\medskip

Now if we combine all 5 cases together, and glue the integral curves from all cases, we see that
the phase portrait at the origin of system~\eqref{e:system} is a saddle (Fig.~\ref{fig7}). 
With this analysis we can now conclude the proof of Proposition \ref{l:2}. Indeed, let $\gamma_1$ and 
$\gamma_2$ be the curves $s=c_\pm t^2$. If $K$ is a small compact not contained in the union of
$\gamma_1$ and $\gamma_2$, then one of the hyperbolas of the characteristic foliation will touch
$K$ at some point. This proves Proposition \ref{l:2}.



\end{document}